\documentclass[12pt]{article}

\usepackage[utf8]{inputenc}
\usepackage[T1]{fontenc}
\usepackage{amsmath, amssymb, amsthm, amsfonts, mathtools, tikz, tikz-cd, xcolor, hyperref, enumitem, mathrsfs}
\usepackage[a4paper, margin=1 in]{geometry}
\usepackage{multicol}
\usepackage{titlesec}
\usepackage[nottoc]{tocbibind}

 \usepackage{libertine} 
 \usepackage[libertine]{newtxmath}

\titleformat{\subsection}[runin]{\bfseries}{\thesubsection}{1em}{}[.]
\titleformat{\subsubsection}[runin]{\itshape}{\thesubsubsection}{1em}{}[.]

\newcommand{\RR}{\mathbb{R}}

\newcommand{\LL}{\mathcal{L}}
\renewcommand{\AA}{\mathcal{A}}
\newcommand{\BB}{\mathcal{B}}
\newcommand{\VV}{\mathcal{V}}

\newcommand{\GG}{\mathcal{G}}
\newcommand{\toto}{\rightrightarrows}
\newcommand{\rar}[1]{\overrightarrow{#1}}
\newcommand{\lar}[1]{\overleftarrow{#1}}
\newcommand{\ddt}{\frac{d}{dt}}
\newcommand{\del}{\partial}
\newcommand{\deldel}[1]{\frac{\del}{\del {#1}}}
\newcommand{\can}{\mathrm{can}}

\newcommand{\FF}{\mathcal{F}}

\newcommand{\mc}[1]{\mathcal{#1}}

\newcommand{\mf}[1]{\mathfrak{#1}}

\DeclareMathOperator{\pr}{pr}

\let\Im\relax

\DeclareMathOperator{\Im}{Im}

\DeclareMathOperator{\id}{Id}

\DeclareMathOperator{\Bis}{Bis}

\DeclareMathOperator{\gr}{gr}

\DeclareMathOperator{\Sing}{Sing} \DeclareMathOperator{\Reg}{Reg}

\makeatletter
\newsavebox{\@brx}
\newcommand{\llangle}[1][]{\savebox{\@brx}{\(\m@th{#1\langle}\)}%
	\mathopen{\copy\@brx\kern-0.5\wd\@brx\usebox{\@brx}}}
\newcommand{\rrangle}[1][]{\savebox{\@brx}{\(\m@th{#1\rangle}\)}%
	\mathclose{\copy\@brx\kern-0.5\wd\@brx\usebox{\@brx}}}
\makeatother
\theoremstyle{plain}
\newtheorem{theorem}{Theorem}[section]
\newtheorem{proposition}[theorem]{Proposition}
\newtheorem{corollary}[theorem]{Corollary}
\newtheorem{lemma}[theorem]{Lemma}

\newtheorem{question}[theorem]{Question}

\theoremstyle{definition}
\newtheorem{df}[theorem]{Definition}
\newtheorem{ex}[theorem]{Example}
\newtheorem{remark}[theorem]{Remark}

\newenvironment{definition}
  {\pushQED{\qed}\df}
  {\popQED\enddf}
\newenvironment{example}
  {\pushQED{\qed}\ex}
  {\popQED\endex}  
  
\newtheorem{innercustomthm}{Claim}
\newenvironment{claim}[1]
  {\renewcommand\theinnercustomthm{#1}\innercustomthm}
  {\endinnercustomthm}

\newtheorem*{claim*}{Claim}
\newtheorem*{notation*}{Notation}

\title{Linearization of Poisson groupoids}
\author{Wilmer Smilde\footnote{wsmilde2@illinois.edu, \textit{University of Illinois at Urbana-Champaign}, Urbana, Illinois, United States}}
\date{}

\begin{document}
\maketitle
\begin{abstract}

Motivated by a search for Lie group structures on groups of Poisson diffeomorphisms, we investigate linearizability of Poisson structures of Poisson groupoids around the unit section. After extending the Lagrangian neighbourhood theorem to the setting of cosymplectic Lie algebroids, we establish that \emph{dual} integrations of triangular bialgebroids are always linearizable. Additionally, we show that the (non-dual) integration of a triangular Lie bialgebroid is linearizable whenever the $r$-matrix is of so-called cosymplectic type. The proof relies on the integration of a triangular Lie bialgebroid to a symplectic LA-groupoid, and in the process we define interesting new examples of double Lie algebroids and LA-groupoids. We also show that the product Poisson groupoid can only be linearizable when the Poisson structure on the unit space is regular. 
\end{abstract}

MSC classes:  58H05, 53D17
\tableofcontents

\section*{Introduction}\addcontentsline{toc}{section}{Introduction}
\newcommand{\Marcut}{M\u{a}rcu\textcommabelow{t}}

In this paper we study linearizability of a Poisson structure $\pi$ on a manifold $M$ around a \textit{Lagrangian} submanifold $L\subset (M, \pi)$, with emphasis on the case that the ambient manifold is a Poisson groupoid, and the submanifold is the unit section. Because of its central role in this paper, we give the definition already here.
\begin{definition}
	Let $(M, \pi)$ be a Poisson manifold. A submanifold $(M, \pi)$ is \textit{Lagrangian} when
	\[
\pi^\sharp\left((TL)^\circ\right)=TL \cap \Im \pi^\sharp.\qedhere
	\]
\end{definition}
In particular, when $L=\{x\}$ consist of just a point, the Lagrangian condition guarantees that $\pi$ vanishes at $x$. In that case, linearization of $\pi$ around $x$ reduces to the \textit{classical linearization problem}, as posed by Weinstein in \cite{weinstein1983}. We can not go on without mentioning Conn's linearization theorem, stating that a Poisson structure is always linearizable around a zero whenever the isotropy Lie algebra is compact and semi-simple \cite{conn1985}. We refer to \cite{crainicfernandes2011}, Appendix 2, for a historical account.

For a Poisson-Lie group $(G, \Pi)$, the identity $e\in G$ is automatically a fixed point of $\Pi$. Linearizability of $(G, \Pi)$ around the identity has been extensively studied, accumulating into several results. For example, the Ginzburg-Weinstein theorem \cite{ginzburgweinstein1992} states that a simply connected dual integration $(G^*, \Pi)$ of a Lie bialgebra $(\mf{g}, \mf{g}^*)$ is \textit{globally} Poisson diffeomorphic to $(\mf{g}^*, \pi_{\mathrm{lin}})$ whenever $\mf{g}$ is compact and semi-simple. A more recent result is due to Alekseev and Meinrenken \cite{alekseevmeinrenken2013}, who showed that dual integrations of \textit{coboundary} Lie bialgebras are always linearizable.

Remarkably, apart from the case of Poisson-Lie groups, very little is known about the linearization of Poisson structures on Poisson groupoids. There could be several reasons for why this is the case. First of all, the linearization problem has mostly been studied for \textit{zeroes} of Poisson structures, while for a Poisson groupoid $(\GG, \Pi)\toto M$ the Poisson tensor is only guaranteed to vanish on the unit space when $M$ has dimension zero.  

Second, only in the recent years, Lagrangian submanifolds in the context of Poisson geometry have gained interest. Our paper shows that Lagrangian submanifolds are indeed fundamental to Poisson geometry, from both the perspective of Poisson diffeomorphisms \cite{smilde2021liegroups} and linearization of Poisson structures. It seems that Weinstein's creed "Everything is a Lagrangian submanifold" \cite{weinstein1981} also holds to some degree in Poisson geometry, at least for the problems considered in this paper and \cite{smilde2021liegroups}. We refer to the PhD-thesis by Aldo Witte \cite{wittephd} for a recent survey on the linearization problem around Lagrangian submanifolds in Poisson geometry.

Interestingly, our methods are much different from the references above. We make heavy use of certain higher structures in differential geometry, such as double Lie algebroids and LA-groupoids, as well as the concept of an $\AA$-Poisson structure $\pi_\AA\in \Gamma(\wedge^2\AA)$, interpreted as a \textit{lift} of a Poisson bivector $\pi$ to a Lie algebroid $\AA\Rightarrow M$. In many cases, the lift is (co)symplectic, and we are able to use symplectic techniques to prove our results. The lifting process of Poisson structures to Lie algebroids has been extensively studied by Klaasse in \cite{klaasse2018}.

Another curious thing to note is that our motivation for studying linearizability of Poisson groupoids comes from the seemingly unrelated problem of finding Lie group structures on groups of Poisson diffeomorphisms. This will be investigated in a separate paper \cite{smilde2021liegroups}, but we will briefly get back to this at the end of this introduction.

\subsection*{Linearization of Poisson structures}

Let $E\to M$ be a vector bundle, with scalar multiplication $m_\lambda:E\to E$, and $\pi$ a Poisson structure on the total space of $E$ for which the zero section $M$ is coisotropic. Then the limit
\[
\pi_{\mathrm{lin}}=\lim_{\lambda\to 0} \lambda m^*_
\lambda(\pi),
\]
exists and is called the \textit{linearization} of $\pi$ around $M$. If $\pi_{\mathrm{lin}}$ is a linear Poisson structure on a vector bundle $E\to M$, meaning that $\lambda m^*_\lambda\pi_{\mathrm{lin}}=\pi_{\mathrm{lin}}$ for all $\lambda\in \RR\setminus \{0\}$, then the zero section is a Lagrangian submanifold of $(E, \pi_{\mathrm{lin}})$ (\cite{smilde2021liegroups}, Proposition 2.2).

More generally, when $(M, \pi)$ is a Poisson manifold and $L\subset (M, \pi)$ is a Lagrangian submanifold, one can use a tubular neighbourhood to transfer $\pi$ to the normal bundle $NL$. Its linearization $\pi_{\mathrm{lin}}$ is independent of the choice of tubular neighbourhood and is in fact corresponding to the Lie algebroid structure on the conormal bundle $N^*L$. If there exists a local Poisson diffeomorphism $(NL, \pi_{\mathrm{lin}})\dashrightarrow (M, \pi)$ that restricts to the identity on $L$, we call $(M, \pi)$ \textit{linearizable} around $L$.
\begin{remark}
The linearization $\pi_{\mathrm{lin}}$ of $(M, \pi)$ as above makes sense for any coisotropic submanifold $C\subset (M, \pi)$. However, only for Lagrangian submanifolds, this is the right notion of linearization, as the zero section in a linear Poisson structure is Lagrangian. For instance, for symplectic leaves, which are also coisotropic, the first-order local model \cite{crainicmarcut2012, vorobiev2000} is not linear.
\end{remark}

We pose the following question. 
\begin{question}
	Let $(M, \pi)$ be a Poisson manifold and $L\subset (M, \pi)$ a Lagrangian submanifold. When is $(M, \pi)$ linearizable around $L$?
\end{question}
When $L=\{x\}$ consists of only one point, the Poisson structure must be zero at $x$, and we recover the classical linearization problem of Poisson structures around zeroes, as mentioned in the first paragraph of this introduction. In the case that $(M, \pi)$ is non-degenerate (that is, symplectic), linearizability is obtained by Weinstein's famous Lagrangian neighbourhood theorem \cite{Weinstein1971}.

In this paper, we prove an extension of Weinstein's result to the setting of cosymplectic Lie algebroids. A \textit{cosymplectic structure} on a Lie algebroid $\AA\Rightarrow M$ consists of closed one-forms $\alpha_1, \dots, \alpha_k$ and a closed 2-form $\omega$ of constant rank $2n$ for which $\omega^n\wedge\alpha_1\wedge\dots\wedge\alpha_k$ is a nowhere vanishing top-form on $\AA$. Underlying a cosymplectic structure on $\AA$, one has an $\AA$-Poisson structure $\pi_\AA\in \Gamma(\wedge^2\AA)$, which ultimately induces a Poisson structure $\pi$ on the underling manifold $M$. 
\begin{theorem}[Theorem \ref{thm:lagrangiannbhd1}]\label{thm:introlagrangiannbhd}
	Let $\AA\Rightarrow M$ be a Lie algebroid with cosymplectic structure $(\alpha_1, \dots, \alpha_k, \omega)$, inducing Poisson structures $\pi_\AA$ on $\AA$ and $\pi$ on $M$. Let $i:L\hookrightarrow M$ be a submanifold satisfying the following three conditions:
	\begin{itemize}[noitemsep, topsep=0em]
		\item the inclusion $i$ is transverse to the anchor of $\AA$,
		\item the pullback algebroid $i^!\AA$ is a Lagrangian subbundle of $(\AA, \pi_\AA)$,
		\item the forms $i^*\alpha_1, \dots, i^*\alpha_k$ are zero in $\Omega^1(i^!\AA)$.
\end{itemize}
 Then $(M, \pi)$ is linearizable around $L$.
\end{theorem} 
The Lagrangian condition together with the vanishing of the pullback one-forms $i^*\alpha_1, \dots, i^*\alpha_k$ ensure that $i^*\omega=0\in \Omega^2(i^!\AA)$.

When $\AA=TM$, the transversality condition is automatically satisfied, and the other two conditions imply that $L$ is a Lagrangian submanifold tangent to the symplectic foliation. When both $\AA=TM$ and $k=0$, we recover Weinstein's Lagrangian neighbourhood theorem from symplectic geometry. Other special cases include Lagrangian submanifolds of log-symplectic and $b^k$-symplectic manifolds \cite{guilleminmirandapires2014, scott2016}, elliptic Poisson manifolds \cite{Cavalcantigualtieri2017} and symplectic foliations, and all of their cosymplectic extensions.

\subsection*{Linearization of Poisson groupoids} A Poisson groupoid \cite{weinstein1988} is a Lie groupoid $\GG\toto M$ with a multiplicative Poisson structure $\Pi$ on $\GG$, meaning that the graph of the multiplication map is coisotropic in $(\GG\times \GG\times \GG, \Pi\times \Pi\times (-\Pi))$. The unit section of a Poisson groupoid $(\GG, \Pi)\toto M$ is always coisotropic. Therefore, identifying the normal bundle $NM$ of $M$ in $\GG$ with the Lie algebroid $\AA$, the dual $\AA^*\Rightarrow M$ comes with a Lie algebroid structure dual to the linearization $\pi_{\mathrm{lin}}$ of $\Pi$ on $\AA\cong NM$. This Lie algebroid structure makes $(\AA, \AA^*)$ into a \textit{Lie bialgebroid}. A detailed account of Poisson groupoids and Lie bialgebroids can be found in \cite{mackenzie2005}, Chapters 11 and 12.

The unit section $M\subset (\GG, \Pi)$ of a Poisson groupoid is in fact always a \textit{Lagrangian submanifold}, as is shown in \cite{smilde2021liegroups}, Proposition 2.2. Therefore, it is natural to ask the following question.
\begin{question}
	Let $(\GG, \Pi)\toto (M, \pi)$ be a Poisson groupoid. Is $(\GG, \Pi)$ linearizable around $M$?
\end{question}
When the answer is positive, we call the Poisson groupoid \textit{linearizable}. Note that this only refers to the \textit{Poisson structure} of the Poisson groupoid, and not to the groupoid itself. This paper is devoted to prove linearizability of two classes of Poisson groupoids. The first concerns \textit{dual} integrations of triangular Lie bialgebroids.
\begin{theorem}[Theorem \ref{thm:dualintegrationlinearizable}]\label{thm:introdualintegration}
Let $(\AA, \AA^*, \pi_\AA)$ be a triangular Lie bialgebroid on $M$. Any Poisson groupoid $(\GG^*, \Pi)\toto M$ integrating $(\AA^*, \AA)$ is linearizable around $M$.
\end{theorem}
In the case that $M$ is a point, the triple $(\AA, \AA^*, \pi_\AA)$ becomes a Lie bialgebra, and linearizability of a dual integration $(\GG^*, \Pi)$ was already known by the work of Alekseev and Meinrenken \cite{alekseevmeinrenken2013}.

Our second linearization result is about normal (i.e. not dual) integrations of triangular Lie bialgebroids. When $(\AA, \AA^*, \pi_\AA)$ is a triangular Lie bialgebroid, and $\GG\toto M$ a groupoid integrating $\AA\Rightarrow M$, the Poisson structure $\lar{\pi_\AA}-\rar{\pi_\AA}$ makes $(\GG, \lar{\pi_\AA}-\rar{\pi_\AA})\toto M$ into a Poisson groupoid integrating $(\AA, \AA^*)$ \cite{liuxu1996}. The $\AA$-Poisson structure (or $r$-matrix) $\pi_\AA$ is of \textit{$\AA$-cosymplectic type} when it underlies a $k$-cosymplectic structure on $\AA$ (Definition \ref{def:cosymplecticstructure}).

\begin{theorem}[Theorem \ref{thm:cosymplecticpairgroupoidlinearizable}]\label{thm:introcosymplecticgroupoidlinearizable}
	Let $(\AA, \AA^*, \pi_\AA)$ be a triangular Lie bialgebroid for which $\pi_\AA$ is of $\AA$-cosymplectic type. If $\GG\toto M$ is an integration of $\AA$, the Poisson groupoid $(\GG, \lar{\pi_\AA}-\rar{\pi_\AA})\toto M$ is linearizable around $M$. 
\end{theorem}
Note that Theorem \ref{thm:introdualintegration} and Theorem \ref{thm:introcosymplecticgroupoidlinearizable} coincide when $\pi_\AA$ is non-degenerate. 

The proof of Theorem \ref{thm:introdualintegration} relies, besides the Lagrangian neighbourhood theorem for symplectic Lie algebroids, on the notion of a symplectic LA-groupoid. More specifically, we show that every Poisson Lie algebroid $(\AA, \pi_\AA)\Rightarrow M$ (i.e. a triangular Lie bialgebroid) integrates to a (local) symplectic LA-groupoid, similar to the integration of a Poisson manifold $(M, \pi)$ to a (local) symplectic groupoid. Its construction consists of two steps already interesting on their own.

The first step requires some terminology. Let $\AA\Rightarrow M$ and $\BB\Rightarrow M$ be Lie algebroids. We say that $\BB$ is \textit{anchored to $\AA$} when there is an identity-covering Lie algebroid morphism $\rho_\BB^\AA:\BB\to \AA$. Next, letting $p_\BB:\BB\to M$ be the projection, we consider the pullback Lie algebroid $p^!_\BB \AA\Rightarrow \BB$ of $\AA$ to $\BB$, as well as, when $\GG\toto M$ is a Lie groupoid integrating $\AA$, the pullback groupoid $p^!_\BB\GG\toto \BB$. The canonical morphisms $p_\BB^!\AA\to \AA$ and $p^!_\BB\GG\to \GG$ are the projections of vector bundle structures. The presence of an anchor $\rho_\BB^\AA:\BB\to \AA$ makes it possible to prolong the Lie algebroid structure on $\BB$ to these vector bundles. 
\begin{theorem}[Theorems \ref{thm:prolongationdouble} and \ref{thm:prolongationLAgroupoid}]\label{thm:introdoubleliealgebroid}
	Let $p_\AA:\AA\Rightarrow M$ and $p_\BB:\BB\Rightarrow M$ be Lie algebroids. Suppose that $\BB$ is anchored to $\AA$ by a Lie algebroid morphism. Then $(p^!_\BB\AA; \BB, \AA;M)$ is a double Lie algebroid. 
	If $\GG\toto M$ is an integration of $\AA$, then $(p^!_\BB \GG; \BB, \GG; M)$ is an LA-groupoid. 
	\[
	\begin{tikzcd}
		p^!_\BB \AA \arrow[r, Rightarrow] \arrow[d, Rightarrow] & \AA \arrow[d, Rightarrow] \\
		\BB \arrow[r, Rightarrow] & M
	\end{tikzcd}
\quad\quad\quad\quad\quad\quad\quad
	\begin{tikzcd}
		p^!_\BB \GG\arrow[r, Rightarrow] \arrow[d, shift left] \arrow[d, shift right] &\GG \arrow[d, shift left] \arrow[d, shift right] \\
		\BB \arrow[r, Rightarrow] & M
	\end{tikzcd}
	\]
\end{theorem}
In both cases, the horizontal algebroid structure on $p_\BB^!\AA$ or $p^!_\BB\GG$ is called the \textit{prolongation algebroid structure}.

\begin{remark}
	In the case that $\AA=TM$ is the tangent bundle, the pullback $(p^!_\BB\AA; \BB, \AA;M)$ reduces to the tangent bundle $(T\BB; \BB, TM; M)$ as a double Lie algebroid. It is in fact very useful to think of $p^!_\BB\AA$ as an $\AA$-tangent version of the tangent bundle of $\BB$. 
\end{remark}

Let $(\AA, \AA^*, \pi_\AA)$ be a triangular Lie bialgebroid, and let $p_*:\AA^*\to M$ be the projection. The map $\pi_\AA^\sharp$ is a Lie algebroid morphism and thus anchors $\AA^*$ to $\AA$, so that $p_*^!\AA$ becomes a double Lie algebroid by Theorem \ref{thm:introdoubleliealgebroid}. Additionally, the Lie algebroid $p_*^!\AA\Rightarrow \AA^*$ comes with canonical symplectic form $\omega_{\can}$ that is infinitesimally multiplicative (IM) with respect to the prolongation Lie algebroid structure in $p_*^!\AA$ (Proposition \ref{prop:omegacanim}). To proceed to the groupoid case, we integrate the closed IM canonical symplectic form to a closed multiplicative form on an integrating LA-groupoid. This is possible by the following result that is proved in the Appendix.
\begin{theorem}[Corollary \ref{cor:multiplicativecochainiso}]
	Let $(\VV; \AA, \GG;M)$ be a (local) source 1-connected LA-groupoid with double Lie algebroid $(\AA\VV; \AA, \AA\GG; M)$. Then the one-to-one correspondence
	\[
	\left\{ \substack{\mbox{(Germs of) multiplicative}\\ \mbox{forms on $\VV\Rightarrow \GG$} } \right\} \xleftrightarrow{1:1} \left\{ \substack{\mbox{Infinitesimally multiplicative}\\
		\mbox{ forms on $\AA\VV\Rightarrow \AA\GG$} } \right\}
	\]
obtained by differentiation is an isomorphism of cochain complexes. 
\end{theorem}
Let $(\AA, \AA^*, \pi_\AA)$ be a triangular Lie bialgebroid. By integrating the canonical symplectic form on $p_*^!\AA\Rightarrow \AA^*$, it follows that an integration $\GG^*\toto M$ of $\AA^*$ carries a (local) multiplicative Poisson structure integrating the bialgebroid $(\AA^*, \AA)$ that comes from a (local) symplectic LA-groupoid $(\VV, \Omega)$. 
\[
\begin{tikzcd}
	(p^!_*\AA, \omega_{\mathrm{can}}) \arrow[r, Rightarrow] \arrow[d, Rightarrow] & \AA \arrow[d, Rightarrow] \\
	(\AA^*, \pi_{\mathrm{lin}}) \arrow[r, Rightarrow] & M
\end{tikzcd}
\rightsquigarrow
\begin{tikzcd}
	(\VV, \Omega) \arrow[d, Rightarrow] \arrow[r, shift right] \arrow[r, shift left] & \AA \arrow[d, Rightarrow]\\
	(\GG^*, \Pi) \arrow[r, shift left] \arrow[r, shift right] & M
\end{tikzcd}
\]
The Lagrangian neighbourhood theorem for (co)symplectic Lie algebroids (Theorem \ref{thm:introlagrangiannbhd}) yields a linearization of $(\GG^*, \Pi)$ around $M$, proving Theorem \ref{thm:introdualintegration}. 
\begin{remark}[Symplectic groupoids] In the case that $\AA=TM$, we recover that every Poisson manifold integrates to a local symplectic groupoid, as we will briefly explain. A Poisson structure $\pi$ on $M$ makes $(TM, T^*M, \pi)$ into a triangular Lie bialgebroid. The pullback algebroid $p_*^!TM \Rightarrow T^*M$ coincides with the tangent bundle $T(T^*M)\Rightarrow T^*M$, which comes with the usual canonical symplectic form $\omega_{\mathrm{can}}$. This form is \textit{infinitesimally multiplicative} on the prolongation Lie algebroid $T(T^*M)\Rightarrow TM$, and therefore integrates to a symplectic form $\Omega$ on the prolongation LA-groupoid $T\GG^*\toto TM$, whenever $\GG^*\toto M$ is a (local) source 1-connected integration of $T^*M\Rightarrow M$. We recover that every Poisson manifold $(M, \pi)$ integrates to a (local) symplectic groupoid. 
\end{remark}

\subsection*{Relation with Poisson diffeomorphisms} The results in this paper are used in \cite{smilde2021liegroups} to obtain (infinite-dimensional) Lie group structures on groups of Poisson diffeomorphisms. We will briefly explain the connection between these two seemingly unrelated topics.

Let $(M, \pi)$ be a Poisson manifold. It is well-known that a diffeomorphism $\varphi:M\to M$ is a Poisson map if and only if its graph $\gr_\varphi$ is coisotropic in $(M\times M, \lar{\pi}-\rar{\pi})$. Since a Poisson diffeomorphism must send each symplectic leaf symplectomorphically onto its image, its graph is not just coisotropic, but also Lagrangian. It is often useful to think of the product $(M\times M, \lar{\pi}-\rar{\pi})$ as a Poisson groupoid over $(M, \pi)$. This way, one can reinterpret Poisson diffeomorphisms as coisotropic bisections. 

More generally, if $(\GG, \Pi)\toto (M, \pi)$ is a Poisson groupoid, then its coisotropic bisections $\Bis(\GG, \Pi)$ form a subgroup of the full bisection group $\Bis(\GG)$. The group of coisotropic bisections $\Bis(\GG, \Pi)$ always acts on the base manifold $(M, \pi)$ by Poisson diffeomorphisms. While $\Bis(\GG)$ is always an infinite-dimensional Lie group, there is no known or obvious Lie group structure on $\Bis(\GG, \Pi)$. 

In \cite{smilde2021liegroups}, we show the following theorem.
\begin{theorem}[\cite{smilde2021liegroups}, Theorem 2.11]\label{thm:introcoisotropicbisections}
	Let $(\GG, \Pi)\toto M$ be a Poisson groupoid with Lie bialgebroid $(\AA, \AA^*)$. Assume that $(\GG, \Pi)$ is \emph{linearizable} around $M$. Then the group of coisotropic bisections is a regular embedded Lie subgroup of $\Bis(\GG)$ with Lie algebra $\Gamma_c(\AA, d_{\AA^*})=\{ v\in \Gamma_c(\AA): d_{\AA^*}v=0\}$, where $d_{\AA^*}$ is the differential of the Lie algebroid $\AA^*\Rightarrow M$.
\end{theorem}
From the perspective of Poisson diffeomorphisms, the pair Poisson groupoid is particularly relevant. However, its linearizability puts heavy restrictions on $(M,\pi)$, as we will show in the current paper.  
\begin{theorem}[Theorem \ref{thm:nonlinearizable}]
	Let $(M, \pi)$ be a Poisson manifold. If $(M\times M, \lar{\pi}-\rar{\pi})$ is linearizable around the diagonal, then $\pi$ must be regular. 
\end{theorem}
This shows that, in order to apply Theorem \ref{thm:introcoisotropicbisections}, one is forced to look for suitable Poisson groupoids which `compute' (sub)groups of Poisson diffeomorphisms of the base, which hopefully turn out to be linearizable. We refer to \cite{smilde2021liegroups}, where this is further investigated.

\subsection*{Acknowledgements} This work, combined with \cite{smilde2021liegroups}, is based on the author's master's thesis at Utrecht University. I want to thank my master's thesis advisor Ioan \Marcut{} for his guidance. I am also in debt to Aldo Witte for our many conversations on the linearization problem. Finally, I want to thank Rui Loja Fernandes for valuable comments on an early version on this paper and Alejandro Cabrera for providing useful references regarding multiplicative forms. 
Declarations of interest: none.

\section{Lagrangian submanifolds in Poisson geometry}

\subsection{Poisson geometry and Lagrangian submanifolds}

A \textit{Poisson structure} on a manifold $M$ is a bivector $\pi\in \Gamma(\wedge^2TM)$ satisfying $[\pi, \pi]=0$. A Poisson map $f:(M, \pi_M)\to (N, \pi_N)$ is a smooth map $f:M\to N$ so that $\pi_N$ is $f$-related to $\pi_M$, meaning that $\wedge^2 T_xf (\pi_{M, x})=\pi_{N,f(x)}$ for all $x\in M$.

When $(M, \pi)$ is a Poisson manifold, the bivector $\pi$ induces a map $\pi^\sharp: T^*M \to TM$, given by $\beta(\pi^\sharp(\alpha))=\pi(\alpha, \beta)$, for $\alpha, \beta\in T^*M$. Together with a Lie bracket on $\Omega^1(M)$ defined by the formula
\[
[\alpha, \beta]_\pi=\LL_{\pi^\sharp(\alpha)}(\beta)-\LL_{\pi^\sharp(\beta)}(\alpha)-d(\pi(\alpha, \beta)),
\]
the cotangent bundle $(T^*M, \pi^\sharp, [\cdot, \cdot]_\pi) $ becomes a Lie algebroid, called the \textit{cotangent algebroid} associated to $(M, \pi)$. 

The leaves of the cotangent algebroid $(T^*M, \pi^\sharp, [\cdot, \cdot]_\pi)$ are called the \textit{symplectic leaves} of $(M, \pi)$, for the reason that every leaf $S$ carries a symplectic structure induced by $\pi$:
\[
\omega_S(X, Y)=-\pi(\alpha, \beta), 
\]
with $X, Y\in TS=\operatorname{im}\pi^\sharp\vert_S$ and $\alpha, \beta\in T^*M$ satisfying $\pi^\sharp(\alpha)=X, \pi^\sharp(\beta)=Y$.

If $C$ is a submanifold of $(M, \pi)$, the \textit{Poisson orthogonal} of $C$ is given by $(TC)^{\bot_{\pi}}=\pi^\sharp((TC)^\circ)$, where $(TC)^\circ\subset T^*M$ is the annihilator of $TC$. 

\begin{definition}
	A submanifold $C$ of $(M, \pi)$ is \textit{coisotropic} when $(TC)^{\bot_\pi}\subset TC$. Equivalently, $\pi(\alpha_1, \alpha_2)=0$ for all $\alpha_1, \alpha_2\in (TC)^\circ$. 
\end{definition}

\begin{definition}[\cite{vaisman1994}, Remark 7.8]\label{def:lagrangian}
	Let $(M, \pi)$ be a Poisson manifold. A submanifold $L\subset M$ is \textit{Lagrangian} when
	\[
	(TL)^{\bot_\pi}=TL\cap \operatorname{im}\pi^\sharp.\qedhere
	\]
\end{definition}

\subsection{$\AA$-Poisson geometry and $\AA$-Lagrangians}
Recall that a Lie algebroid $\AA\Rightarrow M$ is a vector bundle $\AA\to M$ together with an anchor $\rho_\AA:\AA\to TM$ and a Lie bracket on $\Gamma(\AA)$ subject to the Leibniz rule. It induces a Schouten-Nijenhuis bracket on $\Gamma(\wedge^\bullet\AA)$ in the exact same way as for the tangent bundle. The bracket is natural in the sense that if $(\Phi, \varphi):\AA\to \AA'$ is a morphism of Lie algebroids covering $\varphi:M\to M'$, and if $\vartheta_1, \vartheta_2\in \Gamma(\wedge^\bullet \AA)$ are $\Phi$-related to $\vartheta_1', \vartheta_2'\in \Gamma(\wedge^\bullet \AA')$, then the bracket $[\vartheta_1, \vartheta_2]$ is $\Phi$-related to $[\vartheta_1', \vartheta_2']$. 

\begin{definition}
	Let $\AA\Rightarrow M$ be a Lie algebroid. An $\AA$-Poisson structure is a section $\pi_\AA\in \Gamma(\wedge^2\AA)$ such that $[\pi_\AA, \pi_\AA]=0$. In this case, $(\AA, \pi_\AA)\Rightarrow M$ is referred to as a \textit{Poisson Lie algebroid} (or just \textit{Poisson algebroid}).
\end{definition}
\begin{remark}\label{rk:exactbialgebroids}
	Any section $\pi_\AA\in \Gamma(\wedge^2\AA)$ induces a bracket on $\Gamma(\AA^*)$ by
	\[
	[ \alpha, \beta]_{\pi_\AA}=\LL_{\pi^\sharp_\AA(\alpha)}(\beta)-\LL_{\pi^\sharp_\AA(\beta)}(\alpha)-d_\AA \pi(\alpha, \beta). 
	\]
	It equips $\AA^*$ with the structure of a Lie algebroid with anchor $\rho_{\AA^*}:=\rho_\AA\circ \pi^\sharp_\AA$ if and only if $[[\pi_\AA, \pi_\AA], \vartheta]=2[\pi_\AA, [\pi_\AA, \vartheta]]=0$ for all $\vartheta\in \Gamma(\wedge^\bullet\AA)$ \cite{liuxu1996}, in which case $(\AA, \AA^*, \pi_\AA)$ is called an \textit{exact Lie bialgebroid}. It is a simultaneous generalization of coboundary Lie bialgebras and Poisson structures on $TM$.  
	
	The map $\pi^\sharp_\AA:\AA^*\to \AA$ is bracket preserving if and only if $\pi_\AA$ is $\AA$-Poisson. In this case, $(\AA, \AA^*, \pi_\AA)$ is a \textit{triangular Lie bialgebroid}, defined by Mackenzie and Xu in \cite{mackenziexu1994}. Therefore, triangular Lie bialgebroids and Poisson algebroids are the same objects. 
\end{remark}
By naturality of the Schouten-Nijenhuis bracket, an $\AA$-Poisson stucture $\pi_\AA$ induces a Poisson bivector $\pi:=\left(\wedge^2\rho\right)(\pi_\AA)$ on $M$. They are related by the following diagram.
\[
\begin{tikzcd}
	\AA^* \arrow[r, "\pi_\AA^\sharp"] & \AA \arrow[d, "\rho"]\\
	T^*M \arrow[u, "\rho^*"] \arrow[r, "\pi^\sharp"] & TM
\end{tikzcd}
\]
Explicitly, $\pi$ is given by
\[
\pi(\alpha, \beta)=\pi_\AA(\rho_\AA^*\alpha, \rho^*_\AA \beta) \quad \mbox{ for $\alpha, \beta\in T^*M$.}
\]
Accordingly, we refer to $\pi_\AA$ as a \textit{lift} of $\pi$ to $\AA$. 

\begin{definition}
	Let $\AA\Rightarrow M$ be a Lie algebroid. A \textit{symplectic structure} on $\AA$, or an $\AA$-symplectic structure, is a closed, non-degenerate two form $\omega_\AA\in \Omega^2(\AA)$. In this case, we call $(\AA, \omega_\AA)\Rightarrow M$ a \textit{symplectic Lie algebroid}.
\end{definition}
\begin{lemma}
Let $\AA\Rightarrow M$ be a Lie algebroid. Then there is a one-to-one correspondence between non-degenerate $\AA$-Poisson structures and symplectic structures on $\AA$, given by inversion. 
\end{lemma}

\subsubsection{Lie algebroid transversals}\label{subsubsec:liealgebroidtransversals} We introduce the following class of submanifolds.
 
\begin{definition}
	Let $\AA\Rightarrow M$ be a Lie algebroid. A submanifold $L\subset M$ is called an \textit{$\AA$-transversal} when it is transverse to the anchor:
	\[
	T_xL+ \rho_\AA(\AA_x)=T_xM \quad \mbox{ for all $x\in L$.}\qedhere
	\]
\end{definition}
For an $\AA$-transversal $i:L\hookrightarrow M$, the pullback $i^!\AA\Rightarrow L$ always exists, and can be regarded as a subalgebroid of $\AA\Rightarrow M$. As a vector bundle, it is given by $i^!\AA=\rho^{-1}(TL)$. 

We define the \textit{$\AA$-normal bundle of $L$} to be the quotient $N_\AA L= \AA\vert_L / i^!\AA$. It is canonically isomorphic to $NL$, the normal bundle of $L$ in $M$.
\[
\begin{tikzcd}
	0 \arrow[r] & i^!\AA \arrow[d] \arrow[r] & \AA\vert_L \arrow[d] \arrow[r]& N_\AA L \arrow[d, dashed] \arrow[r] & 0\\
	0 \arrow[r] & TL \arrow[r] & TM\vert_L \arrow[r] & NL \arrow[r] & 0.
\end{tikzcd}
\]
More precisely, the anchors $i^!\AA\to TL$ and $\AA\to TM$ induce a map $N_\AA L\to NL$ which is injective because $i^!\AA$ is the pullback, and surjective because $L$ is transverse to $\AA$. 

\begin{definition}
	Let $(\AA, \pi_\AA)\Rightarrow M$ be a Poisson Lie algebroid. An $\AA$-transversal $i:L\hookrightarrow M$ is \textit{coisotropic for $(\AA, \pi_\AA)$} when $i^!\AA$ is a coisotropic subbundle of $\AA$:
	\[
	\left( i^!\AA\right)^{\bot_{\pi_\AA}}:=\pi_\AA^\sharp\left( \left(i^!\AA\right)^\circ\right)\subset i^!\AA.
	\]
	It is \textit{Lagrangian} for $(\AA, \pi_\AA)$ when $\left(i^!\AA\right)^{\bot_{\pi_\AA}}=\left(i^!\AA\right)\cap \operatorname{im}\pi_\AA^\sharp$.
\end{definition} 
When $\AA=TM$, every submanifold is automatically transverse, and coisotropic (Lagrangian) transversals of $(TM, \pi)$ are the same as coisotropic (Lagrangian) submanifolds of $(M, \pi)$. 
\begin{lemma}\label{lem:coisotropictransversal}
	Let $(\AA, \pi_\AA)\Rightarrow M$ be a Poisson Lie algebroid with underlying Poisson manifold $(M, \pi)$. An $\AA$-tranversal $i:L\hookrightarrow M$ is coisotropic for $(\AA, \pi_\AA)$ if and only if $L$ is a coisotropic submanifold of $(M, \pi)$. 
	
	When $L$ is $\AA$-Lagrangian, then it is a Lagrangian submanifold of $(M, \pi)$. 
\end{lemma}
\begin{remark}
	The converse of the Lagrangian statement is false in general, even when $\pi_\AA$ is non-degenerate. For a concrete counterexample, consider a transitive symplectic Lie algebroid $(\AA, \omega)\Rightarrow M$ inducing a Poisson structure $\pi$ on $M$. The diagonal $\Delta$ is Lagrangian in $(M\times M, \lar{\pi}-\rar{\pi})$, but the Lie algebroid $\Delta^!(\AA\times \AA)$ can only be Lagrangian in $(\AA\times \AA, \rar{\omega}-\lar{\omega})$ when $\operatorname{rk}\AA=\dim M$ for dimensional reasons. 
\end{remark}
\begin{proof}[Proof of Lemma \ref{lem:coisotropictransversal}]
	As discussed above, the anchor of $\AA$ induces an isomorphism $\rho:N_\AA L \to NL$, which in turn gives an isomorphism $\rho^*:(TL)^\circ\to (i^!\AA)^\circ$ between the annihilators. Note that $\pi^\sharp\left((TL)^\circ\right)=\rho\circ\pi_\AA^\sharp\left( \left(i^!\AA\right)^\circ\right)$, so $\pi_\AA^\sharp\left( \left(i^!\AA\right)^\circ \right) \subseteq \rho^{-1}(TL)=i^!\AA$ if and only if $\pi^\sharp\left((TL)^\circ\right)\subseteq TL$. The conclusion follows.
	
	Suppose that $i^!\AA$ is Lagrangian for $\pi_\AA$. Then
	\[
	\pi^\sharp\left( (TL)^\circ\right)=\rho\circ\pi_\AA^\sharp \left(\left(i^!\AA\right)^\circ \right)=\rho\left( \left(i^!\AA\right)\cap \operatorname{im}\pi^\sharp_\AA\right)=TL\cap \operatorname{im}\pi^\sharp,
	\]
	which means that $L$ is Lagrangian for $(M, \pi)$. In the last step, we used again explicitly that $\rho^{-1}(TL)=i^!\AA$.  
	\end{proof}

\subsection{Poisson Lie algebroids of cosymplectic type} We introduce cosymplectic structures on Lie algebroids for two reasons. First, they naturally show up in the proof of Theorem \ref{thm:cosymplecticpairgroupoidlinearizable} (even in the case of $\AA=TM$), and second, they provide interesting examples of Poisson manifolds. 
\begin{definition}\label{def:cosymplecticstructure}
	A cosymplectic structure of type $k$, or $k$-cosymplectic structure, on a Lie algebroid $\AA\Rightarrow M$ of $\operatorname{rank}\AA=2n+k$ consists of closed one-forms $\alpha_1, \dots, \alpha_k\in \Omega^1(\AA)$, called the \textit{defining one-forms}, and a closed two-form $\omega\in \Omega^2(\AA)$ of constant rank $2n$ such that $\omega^n\wedge\alpha_1\wedge\dots\wedge\alpha_k$ is nowhere vanishing.
\end{definition}
The non-vanishing condition is equivalent to the cosymplectic flat map 
\[
\flat: \AA\to \AA^*, \quad v\mapsto \iota_v\omega+\sum_{i=1}^k \alpha_i(v)\alpha_i
\]
begin an isomorphism. 
\begin{definition}
	Let $\AA\Rightarrow M$ be a Lie algebroid with cosymplectic structure $(\alpha_1, \dots, \alpha_k, \omega)$. The $i$-th Reeb section $R_i\in \Gamma(\AA)$ is given by $R_i=\flat^{-1}(\alpha_i)$. 
\end{definition}
A $k$-cosymplectic structure $(\alpha_1, \dots, \alpha_k, \omega)$ on $\AA\Rightarrow M$ induces a $(k-1)$-cosymplectic structure $(\alpha_1, \dots, \alpha_{k-1},\omega)$ on the Lie algebroid $\ker \alpha_k$. In particular, $\omega$ restricts to a symplectic structure on the $\AA$-foliation $\FF=\cap_{i=1}^k \ker \alpha_i$. The underlying $\AA$-Poisson structure $\pi_\AA$ is determined by requiring that the following diagram commutes.
\[
\begin{tikzcd}
	\FF^* & \FF \arrow[l, "\omega^\flat"'] \arrow[d] \\
	\AA^* \arrow[u] \arrow[r, "\pi^\sharp_\AA"] & \AA.
\end{tikzcd}
\]
\begin{definition}
	Let $(\AA, \pi_\AA)\Rightarrow M$ be a Poisson algebroid. We call $\pi_\AA$ of \textit{$\AA$-cosymplectic type} if there exists a cosymplectic structure on $\AA$ inducing $\pi_\AA$.
\end{definition}

We are interested in the following Lagrangian submanifolds of cosymplectic Lie algebroids. 

\begin{definition}
	Let $(\AA, \alpha_1, \dots, \alpha_k, \omega)\Rightarrow M$ be a cosymplectic Lie algebroid. An $\AA$-Lagrangian $i:L\hookrightarrow M$ for $(\AA, \pi_\AA)$ is called \textit{minimal} if all the forms $i^*\alpha_1,\dots, i^*\alpha_k$ and $i^*\omega$ vanish in $\Omega^\bullet(i^!\AA)$. 
\end{definition}
The pullback $i^*:\Omega^\bullet(\AA)\to \Omega^\bullet(i^!\AA)$ makes sense because $i$ is transverse to $\AA$. In particular, a minimal Lagrangian transversal is a Lagrangian submanifold of $(M, \pi)$ by Lemma \ref{lem:coisotropictransversal}. 
\begin{remark}
	In the case of $\AA=TM$, minimal Lagrangian submanifolds correspond to Lagrangian submanifolds tangent to the symplectic leaves. In some sense, they are the Lagrangian submanifolds of minimal dimension, hence the name. 
\end{remark}

\section{Pullback algebroids over vector bundles}

Let $p_E:E\to M$ be a vector bundle. The geometry of $TE$ is already very interesting: it is a typical example of a double vector bundle, or even a VB-algebroid, and, varying $E$, there are various canonical isomorphisms involved. This rich structure is used to prolong various geometric objects over vector bundles. For instance, when $E$ is a Lie algebroid, $TE$ can be made into a double Lie algebroid. A fairly complete account of this material can be found in \cite{mackenzie2005}, Chapter 9 and onwards.

In this section, we replace the tangent bundle $TE\Rightarrow E$ of a vector bundle $p_E:E\to M$ by the pullback $p_E^!\AA\Rightarrow E$ along $p_E$ of any Lie algebroid $\AA\Rightarrow M$. It was already observed in \cite{leonmarreromartinez2005} that this object carries a rich structure, analogous to the tangent bundle $TE$. Specifically, in \cite{leonmarreromartinez2005}, the canonical symplectic form on $p_{\AA^*}^!\AA$, as well as a canonical involution on $p_{\AA}^!\AA$ are introduced and studied.

Our approach builds further upon \cite{leonmarreromartinez2005}. First of all, our treatment is entirely coordinate-free. Secondly, we establish that $p_\BB^!\AA$ is a double Lie algebroid, whenever $\BB\Rightarrow M$ is a Lie algebroid anchored to $\AA\Rightarrow M$. This leads to interesting applications later in the paper. 

We assume familiarity with several higher structures in differential geometry, such as double vector bundles, VB-groupoids and -algebroids, and related objects (see for instance \cite{bursztyncabrerahoyo2016, mackenzie2005}). The definitions are recalled in Appendix \ref{app:multiplicativeforms}.

\subsection{The pullback algebroid over a vector bundle}
Given a Lie algebroid $\AA\Rightarrow M$ and a vector bundle $p_E:E\to M$, we consider the pullback Lie algebroid of $\AA$ along $p_E$, whose fibre over a point $e\in E$ can be defined explicitly as
\begin{equation}\label{eq:explicitfibre}
	(p_E^!\AA)_e=\{(v, X)\in \AA_{p_E(e)}\times T_eE:\rho_\AA(v)=T_ep_E(X)\}.
\end{equation}
It fits into the square
\[
\begin{tikzcd}
	p_E^!\AA \arrow[r, "T_\AA p_E"] \arrow[d, "p_E^!\rho_\AA"'] & \AA \arrow[d, "\rho_\AA"] \\
	TE \arrow[r, "Tp_E"] & TM
\end{tikzcd}
\]
where $T_\AA p_E$ is the projection onto the first factor in the fibre description above, and the anchor is $p_E^!\rho_\AA$ is the projection onto the second. It inherits a Lie algebroid structure uniquely determined by requiring that $p_E^!\rho_\AA$ and $T_\AA p_E$ are morphisms of Lie algebroids.
We discuss some functorial properties of this construction.
\begin{itemize}[noitemsep, topsep=0em]
	\item If $p_E:E\to M$ is a vector bundle and $\varphi:\AA\to \BB$ a Lie algebroid morphism covering the identity on $M$, then it induces a Lie algebroid morphism $p^!_E\varphi: p_E^!\AA\to p^!_E\BB$ by $p^!_E\varphi(v, X)=(\varphi(v), X)$. If $\varphi=\rho_\AA:\AA\to TM$ is the anchor, then the map $p_E^!\rho_\AA$ coincides with the one above, after identifying $TE\cong p_E^!TM$.
	\item Given another vector bundle $p_F:F\to M$ and a (not necessarily linear) map $\varphi:F\to E$ satisfying $p_F=p_E\circ \varphi$, there is an induced Lie algebroid morphism $T_\AA\varphi:p_F^!\AA\to p_E^!\AA$ covering $\varphi:F\to E$. Explicitly, on an element $(v, X)\in (p_F^!\AA)_f$ it is given by $(v, T_f\varphi(X))\in (p_E^!\AA)_{\varphi(f)}$, with $f\in F$.
\end{itemize}
A particular, but important case is when $\varphi=m_\lambda:E\to E$ is the scalar multiplication by $\lambda\in \RR$. It induces a map $T_\AA m_\lambda:p_E^!\AA\to p_E^!\AA$ that preserves the fibres of $T_\AA p_E:p_E^!\AA\to \AA$, giving it the structure of a vector bundle. Observe that $p^!_E\AA\to \AA$ is canonically isomorphic to $\rho_\AA^*(T\AA\to TM)$ as a vector bundle over $\AA$.

\begin{proposition}
	With the structures described above, the square
	\[
	\begin{tikzcd}
		p^!_E\AA \arrow[r, "T_\AA p_E"] \arrow[d, Rightarrow] & \AA \arrow[d, Rightarrow] \\ 
		E \arrow[r] & M
	\end{tikzcd}
	\]
	becomes a VB-algebroid, whose core is given by $\ker T_\AA p_E\vert_M\cong \ker Tp_E\vert_M\cong E$. 
\end{proposition}
\begin{proof}
	It is clear that the two scalar multiplications on $p^!_E\AA$ commute, so that it becomes a double vector bundle. Moreover, the structure maps $T_\AA p_E$ and $T_\AA m_\lambda$ are Lie algebroid morphisms for the given Lie algebroid structures, making it into a VB-algebroid. 
\end{proof}

\begin{notation*}
	We need to distinguish between the structure maps of the horizontal and vertical bundle structures on $p^!_E\AA$.
	\begin{itemize}[noitemsep, topsep=0em] 
		\item Scalar multiplication on the vertical vector bundle $p^!_E\AA\to E$ is denoted by $\lambda\cdot V$ or just $\lambda V$ for $\lambda\in \RR$ and $V\in p_E^!\AA$, and similarly, addition is indicated by the usual $+$ symbol. By $0_e$ we mean the zero in the fibre above $e\in E$.
		\item The horizontal scalar multiplication is denoted by $T_\AA m_\lambda$, as introduced above. Addition is defined by $T_\AA(+):p_{E\oplus E}^!\AA\to p_E^!\AA$. The zero above $v\in \AA$ is $T_\AA 0(v)$.
		\item Finally, the inclusion into $p^!_E\AA$ of a core-element $e\in \ker Tp_E\vert_M$ is denoted by $\overline{e}$.
	\end{itemize}
\end{notation*}

\begin{remark} In the case that $\AA=TM$, the usual VB-algebroid structure on $TE$ is recovered. It can be very useful to think of $p_E^!\AA$ as an $\AA$-version of the tangent of a vector bundle. In fact, all the canonical isomorphism involving (duals of) $TE$ described in Chapter 9 of \cite{mackenzie2005} carry over to the pullback algebroids $p_E^!\AA$. However, for length reasons, we only consider those appearing in our applications. 
\end{remark}

\subsection{Generators of sections}
There are two ways to obtain a section of $T_\AA p_E:p_E^!\AA\to \AA$ from a section $e\in \Gamma(E)$.
\begin{itemize}[noitemsep, topsep=0em]
	\item The (horizontal) \textit{core lift} of $e$ is the section 
	\begin{equation}\label{eq:corelift}
		\widehat{e}(v)=T_\AA 0(v)+\overline{e(p_\AA(v))}.
	\end{equation}
	Explicitly, in terms of the fibre description in Equation (\ref{eq:explicitfibre}) and the canonical splitting $TE\vert_M=TM\oplus E$, we have $\widehat{e}(v)=(v, \rho_\AA(v)+e(m))\in \AA_{m}\times T_{0_m}E$, with $m=p_\AA(v)$.
	\item The $\AA$-tangent $T_\AA e$ of $e\in \Gamma(E)$ is a section of $T_\AA p_E$, which is explicitly given by $T_\AA e(v)=(v, T_me(\rho_\AA(v)))$, with $v\in \AA_m$.  
\end{itemize} 
\begin{lemma}\label{lem:horizontalsections}
	The space of sections of $p^!_E\AA\to \AA$ is generated, as a $C^\infty(\AA)$-module, by sections of the form $T_\AA e$ and $\widehat{e}$ for $e\in \Gamma(E)$. 
\end{lemma}
Given $e_1, e_2\in \Gamma(E)$, and $f\in C^\infty(M)$, one can show that the following identities hold: 
\[
T_\AA e_1+T_\AA e_2=T_\AA(e_1+e_2), \quad T_\AA(fe_1)=(p_E^* f)\cdot T_\AA e_1+\ell_{d_\AA f}\cdot \widehat{e_1},
\]
where $\ell_{d_\AA f}$ is the linear function on $\AA$ determined by $\ell_{d_\AA f}(v)=d_\AA f(v)=\LL_v(f)$. The addition and multiplication by functions is induced from the structure functions on the bundle $p_E^!\AA\to \AA$.

In general, there is no canonical way to produce from an element $v\in \Gamma(\AA)$ a section of $p_E^!\AA\to E$ (it would involve an $\AA$-connection on $E$), but there are lifts of core sections. 
\begin{itemize}[noitemsep, topsep=0em]
	\item The \textit{vertical lift} of $e\in \Gamma(E)$ is defined by
	\[
	e^\uparrow(x)=T_\AA(+)(0_x, \overline{e(m)})
	\]
	with $x\in E_m$. Explicitly, it can be expressed as $e^\uparrow(x)=(0, e^\uparrow(x))$, where $e^\uparrow(x)$ is the vertical lift of $e$ to $TE$.
\end{itemize}
\subsubsection{Complete lifts} In the case that the vector bundle $E$ is either $\AA$ or $\AA^*$, there is a way to lift sections of $\AA$ to sections of $p_E^!\AA\to E$, called the complete lifts. To simplify the notation, we set $p:\AA\to M$ and $p_*:\AA^*\to M$ to be the projections. 

The flow of a Lie algebroid section $v\in \Gamma(\AA)$ is a path of Lie algebroid isomorphisms $(\varphi_t, \tilde{\varphi}_t):\AA\to \AA$ uniquely determined by the equations
\[
\frac{d}{dt}\varphi^*_t(w)=\varphi^*_t[v, w] \mbox{ for $w\in \Gamma(\AA)$}, \quad \varphi_0=\operatorname{Id}.
\]
It covers the flow $\tilde{\varphi}_t$ of the vector field $\rho_\AA(v)$. 

The path $T\varphi_t:T\AA\to T\AA$ is the flow of some vector field $X_v\in \Gamma(T\AA)$, called the complete lift of $v$ to $T\AA$. It satisfies $Tp\circ X_v=\rho_\AA(v)\circ p$. The \textit{complete lift} $\widetilde{v}$ of a section $v\in \Gamma(\AA)$ to $p^!\AA$ is defined by
\[
\widetilde{v}(w)=(v(p(w)), X_v(w))\in (p^!\AA)_w, \quad \mbox{ for $w\in \AA$.}
\]
When $v, w\in \Gamma(\AA)$ and $f\in C^\infty(M)$, one can verify that
\[
\widetilde{v+w}=\widetilde{v}+\widetilde{w}, \quad  \widetilde{f\cdot v}=(p^*f)\cdot \widetilde{v}+\ell_{d_\AA f}\cdot v^\uparrow, \quad (f\cdot v)^\uparrow=(p^*f)\cdot v^\uparrow.
\]
\begin{lemma}
	Let $\AA\Rightarrow M$ be a Lie algebroid. Then the complete lifts and vertical lifts generate the space of sections of $p^! \AA$ as a $C^\infty(\AA)$-module. 
\end{lemma}

The story regarding $p_*^!\AA\to \AA^*$ is very similar. If $(\varphi_t, \tilde{\varphi_t})$ is the flow of $v\in\Gamma(\AA)$, then the dual $(\varphi^*_{-t}, \tilde{\varphi}_t):\AA^*\to \AA^*$ is the flow of a vector field $H_v\in\Gamma(\AA^*)$, called the \textit{Hamiltonian lift} or \textit{complete lift} satisfying $Tp_*\circ H_v=\rho_\AA(v)\circ p_*$. It can be regarded as a section of $p^!_*\AA$, also called the \textit{Hamiltonian lift}, via 
\[
H_v(\alpha)=(v, H_v(\alpha))\in (p^!_*\AA)_\alpha.
\] 
\begin{lemma}
	The space of sections of $p^!_*\AA\to \AA^*$ is generated, as a $C^\infty(\AA^*)$-module by the Hamiltonian lifts $H_v$ and vertical lifts $\alpha^\uparrow$ for $v\in \Gamma(\AA)$ and $\alpha\in \Gamma(\AA^*)$. 
\end{lemma}

\begin{remark} The Lie algebroid brackets on $p^!\AA\Rightarrow \AA$ and $p_*^!\AA\Rightarrow \AA^*$ are uniquely determined by their values on the complete and vertical lifts, where they take the form (cf. \cite{mackenzie2005}, Example 3.4.8)
	\begin{align*}
		[\widetilde{v}, \widetilde{w}]&=\widetilde{[v, w]}, & [\widetilde{v}, w^\uparrow]&=[v, w]^\uparrow, & [v^\uparrow, w^\uparrow]&=0, 
		\\
		{[H_v, H_w]}&=H_{[v, w]}, & [H_v, \alpha^\uparrow]&=(\LL_v(\alpha))^\uparrow, & [\alpha^\uparrow, \beta^\uparrow]&=0.
	\end{align*}
\end{remark}
\subsection{The canonical symplectic form}\label{subsubsec:canonicalform}
The total space of the cotangent bundle $T^*M$ for any manifold $M$ comes with a canonical symplectic structure $\omega_\can=-d\lambda_{\can}$ that is given by the differential of the canonical one-form.

More generally, given a Lie algebroid $\AA\Rightarrow M$, the total space $\AA^*$ inherits a linear Poisson structure uniquely determined by the equations
\[
\{\ell_v, \ell_v\}=\ell_{[v, w]}, \quad \{ \ell_v, p_*^*f\}=p_*^*(\LL_{\rho(v)}(f)), \quad \{ p_*^*f, p_*^*g\}=0
\]
for $v\in \Gamma(\AA)$ and $f, g\in C^\infty(M)$. This association defines a one-to-one correspondence between Lie algebroid structures on $\AA$ and linear Poisson structures on $\AA^*$ (\cite{mackenzie2005}, Section 10.3).

When $\AA=TM$, the linear Poisson structure on $T^*M$ is precisely the one induced by $\omega_\can$. In fact, we will see that any linear Poisson structure is induced by a canonical symplectic form, that now lives on a Lie algebroid instead of the tangent bundle. 

\begin{definition}
	Let $\AA\Rightarrow M$ be a Lie algebroid. The canonical one-form $\lambda_\can \in \Omega^1(p^!_*\AA)$ is given by
	\[
	\left(\lambda_{\can}\right)_\alpha (V)=\alpha(T_\AA p_*(V))
	\]
	for $V\in (p^!_*\AA)_\alpha$ and $\alpha\in \AA^*$. 
\end{definition}
As a map, a section $\alpha:M\to \AA^*$ is always transverse to $p_*^!\AA$. The inclusion $\alpha^!p^!\AA\to p_*^!\AA$ corresponds to $T_\AA\alpha:\AA\to p_*^!\AA$. 
\begin{proposition}
	The canonical one-form $\lambda_\can\in \Omega^1(p_*^!\AA)$ is uniquely characterized by the property that
	\[
	(T_\AA\alpha)^*\lambda_\can=\alpha, \quad \mbox{for all $\alpha\in \Omega^1(\AA)$.}
	\]
\end{proposition}
We define the \textit{canonical symplectic form} on $p_*^!\AA$ to be $\omega_\can=-d\lambda_\can\in \Omega^2(p^!_*\AA)$. The next theorem justifies the choice of sign.
\begin{theorem}[\cite{leonmarreromartinez2005}, Theorem 3.2 and Proposition 3.11]\label{thm:canonicalform}
	Let $\AA\Rightarrow M$ be a Lie algebroid. The two-form $\omega_\can\in \Omega^2(p^!_*\AA)$ is symplectic. Furthermore, it induces the linear Poisson structure on $\AA^*$ dual to $\AA$ in the sense that the following commutes:
	\[
	\begin{tikzcd}
		(p_*^!\AA)^* & p_*^!\AA \arrow[d] \arrow[l, "\omega_\can^\flat"']\\
		T^*\AA^* \arrow[u] \arrow[r, "\pi^\sharp"] & T\AA^*.
	\end{tikzcd}
	\]
	Equivalently, $\{f, g\}=-\omega_{\can}(X_f, X_g)$ for all $f, g\in C^\infty(\AA^*)$, where $X_f=\left( \omega_\can^\flat\right)^{-1}(d_\AA f)$ and  $X_g=\left( \omega_\can^\flat\right)^{-1}(d_\AA g)$ are the Hamiltonian sections of $\AA$ corresponding to $f, g$. 
\end{theorem}
\begin{proof}
	The two-form $\omega_{\can}$ is certainly closed. For non-degeneracy, we must first compute the value of $\omega_\can$ on a set of generators. Let $v, w\in \Gamma(\AA)$ and $\alpha, \beta\in \Gamma(\AA^*)$ be given. Then
	\begin{align*}
		\omega_\can(H_{v}+\alpha^\uparrow, H_{w}+\beta^\uparrow)&=\lambda_{\can}\left([H_{v}+\alpha^\uparrow, H_{w}+\beta^\uparrow ]\right)\\
		&\phantom{=}-\LL_{H_{v}+\alpha^\uparrow}(\lambda_{\can}(H_{w}))+\LL_{H_{w}+\beta^\uparrow}(\lambda_\can(H_{v}))\\
		&= \ell_{[v, w]}-\LL_{\alpha^\uparrow}(\ell_v)+\LL_{\beta^\uparrow}(\ell_v)-\LL_{H_{v}}(\ell_w)+\LL_{H_{w}}(\ell_v).
	\end{align*}
	One has
	\[
	\LL_{\alpha^\uparrow}(\ell_w)=p_*^*\langle \alpha, w\rangle, \quad \LL_{H_v}(\ell_w)=\ell_{[v, w]}.
	\]
	It follows that
	\begin{equation}\label{eq:canonicalform}
		\omega_{\can}(H_v+\alpha^\uparrow, H_w+\beta^\uparrow)=p_*^*\langle \beta, v\rangle-p_*^*\langle \alpha, w\rangle -\ell_{[v, w]}.
	\end{equation}
	Non-degeneracy of $\omega_\can$ is now obvious.
	
	To compute the Poisson structure induced by $\omega_\can$, notice that
	\[
	\omega_\can^\flat(H_v)=+d\ell_v, \quad \omega_\can^\flat((df)^\uparrow)=-d(p^*_*f).
	\]
	where $d$ is the differential for $\AA$ on the LHS, and the differential for $p_*^!\AA$ on the RHS. Therefore, the Poisson structure induced by $\omega_\can$ satisfies
	\begin{align*}
		\{\ell_v, \ell_w\}&=-\omega_{\can}(H_v,H_w)=\ell_{[v, w]},\\
		\pi(\ell_v, p_*^*f)&=-\omega_\can(H_v, -(df)^\uparrow)=p_*^*(\LL_v(f)),\\
		\pi(p_*^*f, p_*^*g)&=-\omega_{\can}(-(df)^\uparrow, -(dg)^\uparrow)=0.
	\end{align*}
	It follows that $\pi$ coincides with the linear Poisson structure on $\AA^*$.
\end{proof}

\begin{corollary}\label{cor:closedoneformlagrangian}
	Let $\AA\Rightarrow M$ be a Lie algebroid. The image of a section $\alpha:M\to \AA^*$ is coisotropic in $(\AA^*, \pi_{\mathrm{lin}})$ if and only if $d_\AA \alpha=0$, in which case it is a $p_*^!\AA$-Lagrangian transversal, and therefore Lagrangian in $(\AA^*, \pi_{\mathrm{lin}})$.
\end{corollary}
\begin{proof}
	A one-form $\alpha$ is, regarded as a map $\alpha:M\to\AA^*$, always transverse to $p_*^!\AA\to \AA^*$, and $\alpha^!p^!_*\AA=\AA$ has half the rank of $p_*^!\AA$. Therefore, by Lemma \ref{lem:coisotropictransversal}, its image is coisotropic if and only if $\alpha^*\omega_{\can}=-d_\AA \alpha=0$, in which case it is Lagrangian for dimensional reasons. 
\end{proof}
We end with a lemma that describes the interaction of the canonical symplectic form with the horizontal vector bundle structure, providing necessary identities for later calculations. Examples are postponed to Section \ref{subsubsec:lagrangiannbhd}.
\begin{lemma}\label{lem:canonicalformprolongation}
	Let $\AA\Rightarrow M$ be a Lie algebroid. Then $\omega_\can$ is linear with respect to the horizontal structure on $p_*^!\AA$. Moreover, the following identities hold.
	\begin{align*}
		\omega_{\can}(T_\AA \alpha(v), T_\AA\alpha(w))&=(d\alpha)(v, w),\\ \omega_{\can}(T_\AA 0(v), \widehat{\alpha}(w))&=-\omega_{\can}(\widehat{\alpha}(w), T_\AA 0(v))=\langle \alpha, v\rangle
	\end{align*}
	for $\alpha\in \Omega^1(\AA)$ and $v, w\in \Gamma(\AA)$. 
\end{lemma}
\begin{proof}
	Linearity of $\omega_{\mathrm{can}}$ follows from linearity of the canonical one-form:
	\begin{align*}
		\left((T_\AA m_\varepsilon)^*\lambda_{\can}\right)_{\alpha} (V)&=\left(\lambda_{\can}\right)_{\varepsilon\alpha}\left(T_\AA p_*\circ T_\AA m_\varepsilon(V)\right)\\
		&=\langle \varepsilon\alpha, T_\AA p_*(V)\rangle=\varepsilon\left(\lambda_{\can}\right)_\alpha(V),
	\end{align*}
	To show the first equality, note that for $x\in M$, the expression $T_\AA\alpha(v(x))-H_v(\alpha(x))$ defines an element in $\ker T_\AA p_*$, and is therefore vertical. To determine its value, we only have to evaluate it on linear functions. We claim that
	\begin{equation}\label{eq:derivativesection}
		(T_\AA \alpha)(v(x))-H_v(\alpha(x))=\LL_v(\alpha)^\uparrow(\alpha(x)).
	\end{equation}
	Indeed, for $w\in\Gamma(\AA)$, we have
	\[
	(T_\AA\alpha)(v(x))(\ell_w)=\LL_{\rho(v)}\langle \alpha, w\rangle(x), \quad
	H_v(\alpha)(\ell_w)=\ell_{[v, w]}(\alpha(x)),
	\]
	so $\left(T_\AA \alpha(v(x))-H_v(\alpha(x))\right)(\ell_w)=\left\langle\LL_{v}(\alpha), w\right\rangle(x)$. Now the first equality easily follows from Equation (\ref{eq:canonicalform}).
	
	For the second, note that by Equation (\ref{eq:derivativesection}) we have $T_\AA 0(v)=H_v(0)$ for $v\in \Gamma(\AA)$. Thus, recalling that $\widehat{\alpha}(v)=T_\AA0(v)+\overline{\alpha(p(v))}$, it follows that
	\[
	\omega_{\can}\left(T_\AA0(v), \widehat{\alpha}(w)\right)=\omega_{\can}\left(H_v(0), H_w(0)+\alpha^\uparrow(0)\right)=\langle \alpha, v\rangle,
	\]
	finishing the proof.
\end{proof}

\subsection{The canonical $\AA$-involution}

The canonical involution $J:T(TM)\to T(TM)$ can be thought of as a flip of derivatives. An element $X\in T(TM)$ can be represented by a smooth map $h:\RR^2\to M$ such that $X=\deldel{t}\deldel{s}h (0,0)$. Then $J(X)=\deldel{s}\deldel{t}h(0,0)$. 

Similarly, there is an involution $J_\AA:p^!\AA\to p^!\AA$ that is an isomorphism of double vector bundles between $p^!\AA$ and its flip, which is defined by interchanging the $\AA$-derivatives. 

Let $I$ be a small interval around $0$ and let $H:TI\times TI\to \AA$ be a Lie algebroid map, covering $\gamma:I^2\to M$. It can be decomposed as $H=h_1dt+h_2 ds$, where $h_1, h_2:I^2\to \AA$ can be thought of replacements of the partial derivatives of $\gamma$. The map $H$ defines an element of $p^!\AA$ via
\[
\del_{t}\del_sH(0,0)=\left( h_1(0,0), \deldel{t} h_2(0,0)\right)\in (p^!\AA)_{\gamma(0,0)}.
\]
This is well-defined because
\[
\rho_\AA(h_1(0,0))=\deldel{t}\gamma(0,0)=\deldel{t}p\circ h_2(0,0)=Tp\left(\deldel{t}h_2(0,0)\right).
\]
\begin{definition}
	The map $J_\AA:p^!\AA\to p^!\AA$ is defined on $\del_t\del_s H(0,0)$ as
	\[
	J_\AA\left(h_1(0,0), \deldel{t}h_2(0,0)\right)=\left( h_2(0,0), \deldel{s} h_1(0,0)\right)=\del_s\del_tH(0,0).\qedhere	\]
\end{definition}

\begin{theorem}[\cite{leonmarreromartinez2005}, Theorem 4.2]
The map $J_\AA:p^!\AA\to p^!\AA$ is well-defined, an involution, and an isomorphism of double vector bundles between $p^!\AA$ and its flip. 
\end{theorem}
For future reference, we state how it acts on the generating sections of $p^!\AA$. 

\begin{proposition}[\cite{leonmarreromartinez2005}, Theorem 4.4]\label{prop:canonicalinvolution}
	The canonical involution $J_\AA:p^!\AA\to p^!\AA$ is uniquely determined by the identities
	\[
	J_\AA(\widetilde{v})=T_\AA v, \quad J_\AA(v^\uparrow)=\widehat{v}. 
	\]
\end{proposition}

\subsection{Prolongations over Lie algebroids}\label{subsec:prolongationsoverLiealgebroids}

When $\BB\Rightarrow M$ is a Lie algebroid anchored to $\AA$ via a Lie algebroid morphism $\rho_{\BB}^\AA:\BB\to \AA$, then the Lie algebroid structure on $\BB$ can be prolonged to a Lie algebroid structure on $p^!_\BB\AA\Rightarrow \AA$, in exactly the same way as the Lie algebroid structure on $TB\Rightarrow TM$ arises (\cite{mackenzie2005}, Section 9.7). 

\textbf{The anchor.} Set $\rho_{T_\AA}=J_\AA\circ T_\AA(\rho_\BB^\AA):p^!_\BB\AA\to p_\AA^!\AA$, and compose with the anchor $p_\AA^!\AA\to T\AA$. The resulting map is denoted $\varrho_{\BB}:p^!_\BB\AA\to T\AA$, and will be the anchor of the prolongated structure. It is map of double vector bundles, that restricts to $\rho_\BB^\AA$ on the core- and side bundle $\BB$ and the identity on the side bundle $\AA$. 

\textbf{The bracket.} Recall from Lemma \ref{lem:horizontalsections} that $\Gamma_\AA(p^!_\BB\AA)$ is generated by sections of the form $T_\AA b$ and $\widehat{b}$ for $b\in \Gamma(\BB)$. On those generators, the Lie bracket is defined as
\[
[T_\AA b_1, T_\AA b_2]=T_\AA[b_1, b_2], \quad [T_\AA b_1, \widehat{b_2}]=\widehat{[b_1, b_2]}, \quad [\widehat{b_1}, \widehat{b_2}]=0, 
\]
for $b_1, b_2\in \Gamma(\BB)$. In exactly the same way as Theorem 9.7.1 in \cite{mackenzie2005}, we have the following.
\begin{proposition}
	Let $\BB\Rightarrow M$ be an $\AA$-anchored Lie algebroid. The anchor and Lie bracket described above make $p_\BB^!\AA\Rightarrow \AA$ into a $p_\AA^!\AA$-anchored Lie algebroid.
\end{proposition}

\begin{theorem}\label{thm:prolongationdouble}
	Let $\BB\Rightarrow M$ be an $\AA$-anchored Lie algebroid. Then
	\[
	\begin{tikzcd}
		p^!_\BB\AA \arrow[r, Rightarrow] \arrow[d, Rightarrow] & \AA \arrow[d, Rightarrow]\\
		\BB \arrow[r, Rightarrow] & M
	\end{tikzcd}
	\]
	is a double Lie algebroid. 
\end{theorem}
Due to the complicated definition of a double Lie algebroid, a direct proof of Theorem \ref{thm:prolongationdouble} would be quite involved, and significantly lengthen the paper. Instead, we will think of double Lie algebroids as the infinitesimal counterparts of LA-groupoids, as developed by Mackenzie in \cite{mackenzie1992,mackenzie1997}. We will prove in the next section that the above square is the infinitesimal object associated to an LA-groupoid.

\subsubsection{The pullback groupoid over a Lie algebroid} Let us return to the setting in which $\AA\Rightarrow M$ is a Lie algebroid, and $\BB\Rightarrow M$ is a Lie algebroid anchored to $\AA$. Suppose that $\GG\toto M$ is a (local) integration of $\AA$. Then $p_\BB^!\GG\toto \BB$ is a (local) VB-groupoid integrating $p_\BB^!\AA\Rightarrow \BB$. Explicitly, the pullback groupoid is given by
\[
p_\BB^!\GG=\{ (v, g, w)\in \BB\times \GG\times \BB: p_\BB(v)=t(g), p_\BB(w)=s(g)\}
\]
whose source and target maps $\tilde{s}, \tilde{t}:p^!_\BB\GG\to \BB$ are the projections on the first and last $\BB$-components, and multiplication is given by $(v, g, w)\cdot (v', h, w')=(v, gh, w')$ whenever the tuples are composable. As a vector bundle, it is isomorphic go $t^*\BB\oplus s^*\BB\to \GG$.

We now describe the Lie algebroid structure on $p_\BB^!\GG\to \GG$.

\textbf{The anchor.} For $(v, g, w)\in p_\BB^!\GG$, we set the anchor $\varrho:p^!_\BB\GG\to T\GG$  to be
\[
\varrho(v, g, w)=T_{t(g)}R_g (\rho_\BB^\AA(v))+T_{s(g)}L_g T_{s(g)} i (\rho_\BB^\AA(w))\in T_g\GG. 
\]
where we identified $\AA\cong \ker Ts\vert_M$. It can be readily checked to be a morphism of VB-groupoids, covering the map $\rho_\AA^\BB\circ\rho_\BB=\rho_\BB: \BB\to TM$ on the unit spaces and $\rho_\BB^\AA:\BB\to \AA$ on the cores.

\textbf{The Lie bracket.} First, sections of the form $t^*v$ and $s^*v$, with $t^*v(g)=(v(t(g)), g, 0))$ and $s^*v(g)=(0, g, v(s(g))$, generate $\Gamma(p^!_\BB\GG\to \GG)$. We define a Lie bracket on those generators by
\[
[t^*v, t^*w]=t^*[v, w], \quad [s^*v, s^*w]=s^*[v, w], \quad [t^*v, s^*w]=0, 
\]
and extend it via the Leibniz rule. This is possible, because the equations are compatible: if $f\in C^\infty(M)$ and $v, w\in \Gamma(\BB)$, then
\begin{align*}
[t^*v, s^*(fw)]&=	[t^*v, s^*f s^*w]=\LL_{\varrho(t^*u)}(s^*f) s^*w=0,\\
[t^*v, t^*(fw)]&=[t^*v, t^*f t^*w]=\LL_{t^*v}(t^*f) t^*w+(t^*f) t^*[v, w]\\
&=t^* \LL_v(f)t^*w+t^*(f[v,w])=t^*([v, fw]).
\end{align*}
One readily verifies that the anchor $\varrho$ preserves the bracket. The Jacobi identity can be proven from the following claim that can be verified by a straightforward computation. Let 
\[
\operatorname{Jac}(U, V, W)=[[U,V],W]+[[V, W], U]+[[W, U], V].
\]
for $U, V, W\in \Gamma(p^!_\BB\GG\to \GG)$. Then for all $f\in C^\infty(\GG)$ we have
\[
\operatorname{Jac}(U, V, fW)=f\operatorname{Jac}(U, V, W)+\LL_{\varrho\left([U, V]\right)}(f)-\LL_{\left[\varrho(U), \varrho(V)\right]}(f). 
\]
Since the anchor $\varrho$ is already bracket-preserving, it follows that $\operatorname{Jac}$ is $C^\infty(\GG)$-multilinear, and so it vanishes on all sections whenever it is zero on a set of generators. It is readily checked that it vanishes on sections of the form $t^*u$ and $s^*u$, so the above construction gives a Lie algebroid over $\GG$.

\begin{theorem}\label{thm:prolongationLAgroupoid}
	Let $\AA, \BB$ be Lie algebroids over $M$, and let $\GG\toto M$ be an integration of $\AA$. Suppose that $\BB$ is anchored to $\AA$. Then, with the Lie algebroid structure described above, the square
	\[
	\begin{tikzcd}
		p^!_\BB\GG \arrow[r, Rightarrow]\arrow[d, shift left] \arrow[d, shift right] & \GG \arrow[d, shift left]\arrow[d, shift right]\\
		\BB \arrow[r, Rightarrow] & M
	\end{tikzcd}
	\]
	is an LA-groupoid that differentiates to the double Lie algebroid $(p^!_\BB\AA; \BB, \AA; M)$.
\end{theorem}

\begin{proof}
	Following the definition in \cite{mackenzie1997}, the only thing that is left to show is that the groupoid structure maps are Lie algebroid morphisms. The unit map is clearly a Lie algebroid morphism. For the others, we make use of the following lemma.
	\begin{lemma}[Proposition 3.4.8 in \cite{mackenzie2005}] If $(\varphi, f):\AA\to \AA'$ is a vector bundle morphism such that
		\begin{itemize}[noitemsep, topsep=0em]
			\item[1.] $\varphi$ is fibrewise surjective,
			\item[2.] $\rho_{\AA'}\circ \varphi=Tf\circ\rho_\AA$,
			\item[3.] If $v,w\in \Gamma(\AA)$ are $\varphi$-related to $v', w'\in\Gamma(\AA')$, then $[v,w]$ is $\varphi$-related to $[v',w']$,
		\end{itemize}
	then $(\varphi,f)$ is a Lie algebroid morphism.
	\end{lemma}
It follows immediately that the source and target maps are Lie algebroid morphisms. Inversion is also straightforward. Finally, for multiplication, notice that the vector bundle $(p_\BB^!\GG)^{(2)}\to \GG^{(2)}$ is generated by sections of the form
\[
(u,v,w):=\pr_1^*t^*u+\pr_1^*s^*v+\pr_2^*t^*v+\pr_2^*t^*w
\]
for $u,v,w\in \Gamma(\BB)$, with bracket $[(u_1, v_1, w_2), (u_2, v_2, w_2)]=([u_1, u_2], [v_1, v_2], [w_1, w_2])$. It is now clear that the multiplication map, sending $(u, v, w)$ to $t^*u+s^*w$, is bracket preserving. 
\end{proof}

\begin{proof}[Proof of Theorem \ref{thm:prolongationdouble}]
Let $\GG\toto M$ be a (local) integration of $\AA$.  Then by \cite{mackenzie1997}, Definition 1.16, it is enough to show that the Lie algebroid structure on $p^!_\BB\GG\Rightarrow \GG$ differentiates to the prolongation algebroid structure on $p^!_\BB\AA\Rightarrow \AA$. The differentiation procedure in \cite{mackenzie1997} is briefly summarized in Appendix \ref{sec:multiplicativeformsLAgroupoids}. 

It is clear that the anchor of $p^!_\BB\GG\Rightarrow \GG$ differentiates to the anchor of $p^!_\BB\AA\Rightarrow \AA$. For the bracket, notice that the core sections of $p^!_\BB\GG$ are generated by $t^*v$ for $v\in \Gamma(\BB)$. Moreover, for $v\in \Gamma(\BB)$ the pair $(t^*v+s^*v, v)$ is a star section, and star sections of this form, together with the core sections already generate $\Gamma(p^!_\BB\GG\Rightarrow \GG)$. Furthermore, in the notation of Section \ref{sec:multiplicativeformsLAgroupoids}, $\AA(t^*v+s^*v)=T_\AA v$. The bracket obtained by differentiation corresponds with the Lie bracket of the prolongation in Theorem \ref{thm:prolongationdouble} because
\begin{align*}
\AA[t^*v+s^*v, t^*w+s^*w]&=\AA(t^*[v, w]+s^*[v, w])=T_\AA[v,w],\\
[t^*v+s^*v, t^*w]\vert_M&=[v, w]
\end{align*}
and the bracket between core sections are zero for both.
\end{proof}

\section{Linearization of Poisson groupoids}\label{sec:linearizationofpoissongroupoids}

In this section we investigate the linearizability of several classes of Poisson groupoids. Before we do so, we first prove a linearization result for Lagrangian transversals of (co)symplectic Lie algebroids.

\subsection{Lagrangian neighbourhood theorem for cosymplectic Lie algebroids}\label{sec:linearization}

We generalize Weinstein's Lagrangian neighbourhood theorem \cite{Weinstein1971} to the setting of cosymplectic Lie algebroids. The proof relies on a Moser argument, as well as a suitable homotopy operator for Lie algebroid cohomology along transversal submanifolds, which is made possible by means of the splitting theorem for Lie algebroids by Bursztyn, Lima and Meinrenken \cite{Bursztyn2016Splitting}.

\subsubsection{An $\AA$-cosymplectic Moser lemma}
The $\AA$-symplectic version of the lemma below will also appear in \cite{sjamaarpreparation}. 

\begin{lemma}[Cosymplectic Moser lemma]\label{lem:moserlemma}
	Let $\AA\Rightarrow M$ be an algebroid with cosymplectic structures $(\alpha_1, \dots, \alpha_k, \omega)$ and $(\tilde{\alpha}_1,\dots\tilde{\alpha}_k,  \tilde{\omega})$ that agree along a transverse submanifold $L$ in the sense that $\alpha_j\vert_L=\tilde{\alpha}_j\vert_L$ and $\omega\vert_L=\tilde{\omega}\vert_L$. Then there exists a Lie algebroid isomorphism $(\tilde{\varphi}, \varphi): \AA\vert_{U}\to \AA\vert_{\tilde{U}}$ over open neighbourhoods $U$, $\tilde{U}$ of $L$ in $M$ such that the following hold:
	\begin{itemize}[noitemsep, topsep=0em]
		\item $\tilde{\varphi}^*\tilde{\alpha}_j=\alpha_j$ for $j=1, \dots, k$;  
		\item $\tilde{\varphi}^*\tilde{\omega}=\omega+\sum_{j=1}^k d(f^j\alpha_j)$ for some functions $f^j$. 
		\item $\tilde{\varphi}\vert_{\AA\vert_L}$ is the identity.
	\end{itemize} 
	In particular, $(\tilde{\varphi}, \varphi)$ is an $\AA$-Poisson isomorphism for the underlying Poisson structures.
\end{lemma}
The proof requires a few ingredients.
\begin{theorem}[Splitting theorem for Lie algebroid transversals, \cite{Bursztyn2016Splitting}]\label{thm:splitting} Let $\AA\Rightarrow M$ be a Lie algebroid, and $i:L\hookrightarrow M$ a closed, embedded submanifold transverse to $\AA$. Then there exists a neighbourhood $U$ of $L$ in $M$ and a Lie algebroid isomorphism
	\[
	p_{NL}^!i^!\AA\to\AA\vert_U\
	\] 
	covering a tubular neighbourhood $NL\to U$. 
\end{theorem}
The following result can be seen is a direct application of the splitting theorem (see \cite{Weinstein1971}, Equation 3.1, for a similar result for the tangent bundle only; the underlying principles are very much the same).  
\begin{lemma}[Relative Poincar\'e lemma for Lie algebroid transversals]\label{lem:poincare} Let $\AA\Rightarrow M$ be a Lie algebroid, and $i:L\hookrightarrow M$ a submanifold transverse to $\AA$. Let $\omega\in \Omega^{k+1}(\AA)$ be a closed $(k+1)$-form such that $0=i^*\omega\in \Omega^{k+1}(i^!\AA)$. Then there is an open neighbourhood $U$ of $L$ in $M$ and a $k$-form $\varphi\in\Omega^k(\AA\vert_U)$ such that $d\varphi=\omega\vert_U$ and $\varphi\vert_L=0$. If moreover $\omega\vert_L=0$, then all first $\AA$-derivatives of $\varphi$ vanish along $L$.
\end{lemma}
\begin{proof}
	Using the splitting theorem, we can pass on to a tubular neighbourhood, and assume that $M=E$ is the total space of a vector bundle $p:E\to M$ and $\AA=p^!i^!\AA$. For notational purposes, we set $\BB=i^!\AA$. Because $\ker Tp$ sits naturally inside $p^!\BB$, the Euler vector field $\mathcal{E}$ on $E$, being vertical, can be regarded as a section of $p^!\BB$. Its Lie algebroid flow $(\Phi, \varphi)$ is exactly $(T_{\BB} m_{e^t}, m_{e^t})$.

Start with the following:
	\begin{align*}
		\ddt \left( T_\BB m_t\right)^* \omega&= \ddt \left( \Phi_{\log t} \right)^* \omega = \frac{1}{t}\left(\Phi_{\log t}\right)^*\left( \LL_{\mc{E}}(\omega)\right)\\
		&=\frac{1}{t}\left(\Phi_{\log t}\right)^* d(\iota_{\mc{E}} \omega)= d\left( \frac{1}{t} \left(T_\BB m_t\right)^*\left(\iota_{\mc{E}} \omega\right)\right).
	\end{align*}
	The form $(1/t) \left(T_\BB m_t\right)^* (\iota_{\mc{E}} \omega)$ extends smoothly to $t=0$, and for all $t\in [0,1]$, it vanishes along $L$.
	
	Because $T_\BB m_0=T_\BB(i\circ p)=T_\BB i\circ T_\BB p$ factors through the inclusion $T_\BB i:\BB \hookrightarrow p^!\BB$, we can write
	\begin{align*}
		\omega&=\left(T_\BB m_1\right)^*\omega -\left(T_\BB m_0\right)^* \omega=\int_0^1 \ddt \left(T_\BB m_t\right)^* \omega dt\\
		&= d \left(\int_0^1 \frac{1}{t} \left(T_\BB m_t\right)^*(\iota_{\mc{E}} \omega) dt\right)=d\varphi.
	\end{align*}	
	Finally, assume that $\omega\vert_L=0$. Then for all sections $X\in \Gamma(p^!i^!\AA)$ we have $\LL_X(\iota_{\mc{E}}\omega)\vert_L=0$. Therefore, $\LL_X(\varphi)\vert_L=0$, so all the first $\AA$-derivatives of $\varphi$ vanish along $L$.
\end{proof}
\begin{proof}[Proof of Lemma \ref{lem:moserlemma}]
	The proof is a two step Moser argument. First, we adjust the cosymplectic structures so that we can assume that the one-forms agree in a neighbourhood of $L$. 
	In the second step we make the two-forms agree up to the `gauge equivalence' in the statement of the lemma.
	
	\textit{Step 1.} The one-forms $\alpha_j$ and $\tilde{\alpha}_j$ agree along $L$, so by the Poincar\'e Lemma \ref{lem:poincare} there are functions $h_j$ such that $\tilde{\alpha}_j-\alpha_j=dh_j$ in a neighbourhood $U$ of $L$, where $h_j$ and its first $\AA$-derivatives vanish along $L$. 
	
	Consider the one-forms $\alpha_{t,j}=\alpha_j+t(\tilde{\alpha}_j-\alpha_j)=\alpha_0+tdh_j$. Shrinking $U$ if necessary, we can assume that $\alpha_{t,j}$ are linearly independent on $U$ for each $t\in [0,1]$. 
	
	We can find a time-dependent section $X_t\in \Gamma(\AA\vert_U)$, vanishing together with its $\AA$-derivatives along $L$ 
	(meaning that $[X_t, Y]\vert_L=0$ for all $Y\in \Gamma(\AA)$), that satisfies $\alpha_{t,j}(X_t)=-h_j$ for all $j$. Its flow $(\tilde{\varphi}_t, \varphi_t)$ is, after shrinking, defined on $U$ for $t\in [0,1]$. Then
	\[
	\ddt (\tilde{\varphi}_t)^*(\alpha_{t,j})= \tilde{\varphi}_t^*\left(d\alpha_{t,j}(X_t)+dh_j\right)=0,
	\]
	so at $t=1$ we have $\tilde{\varphi}_1^*\tilde{\alpha}_j=\alpha_j$. Because $X_t$ and its first $\AA$-derivatives vanish along $L$, the isomorphism $\tilde{\varphi}_t$ restricts to the identity on $\AA\vert_L$. Therefore, $\tilde{\varphi}_t^*\tilde{\omega}\vert_L=\tilde{\omega}\vert_L$, so the two-forms still agree along $L$.
	
	\textit{Step 2.} By Step 1, we can assume that $\alpha_j=\tilde{\alpha_j}$ for all $j=1, \dots, k$. By the Poincar\'e lemma, there is a one-form $\varphi\in \Omega^{1}(\AA\vert_U)$ with $d\varphi=\tilde{\omega}-\omega$ and $\varphi\vert_L=0$.
	
	As usual, let $\omega_t=\omega+t(\tilde{\omega}-\omega)$, which is non-degenerate on $\FF=\ker \alpha_1\cap\dots \cap \ker \alpha_k$ in a neighbourhood of $L$. Because of this, we can find $Y_t\in \Gamma(\FF)$ with $\iota_{Y_t}\omega_t=-\varphi+ \sum_{j=1}^k(f_t^j\alpha_j)$ for some functions $f_t^j\in C^\infty(U)$. The flow $(\tilde{\psi}_t, \psi_t)$ is defined for $t\in [0,1]$ in a neighbourhood $U$ of $L$, and on this neighbourhood it satisfies
	\begin{align*}
		\ddt(\tilde{\psi}_t)^*\omega_t&=(\tilde{\psi}_t)^* \left( d(\iota_{Y_t}\omega_t) + d\varphi\right)= \sum_{j=1}^kd\left(\tilde{\psi^*_t}\left(f^j_t\right)\tilde{\psi}^*_t\left(\alpha_j\right)\right) =\sum_{j=1}^kd\left(\tilde{\psi^*_t}\left(f^j_t\right)\alpha_j\right), \\
		\ddt(\tilde{\psi}_t)^*\alpha_j&=(\tilde{\psi}_t)^*\left(d(\iota_{Y_t} \alpha_j)\right)=0.
	\end{align*}
	Hence at time $t=1$, we find
	\[
	\tilde{\psi}_1^*\alpha_j=\alpha_j, \quad \tilde{\psi}_1^*\tilde{\omega}=\omega+ \sum_{j=1}^kd\left(\int_0^1 \tilde{\psi}_t^*f_t^j dt \right)\alpha_j.
	\]
	The statement is proven by taking $\tilde{\varphi}=\tilde{\psi}_1\circ\tilde{\varphi}_1$.
\end{proof}
\subsubsection{An $\AA$-cosymplectic Lagrangian neighbourhood theorem}\label{subsubsec:lagrangiannbhd}
The next result is a generalization of Weinstein's Lagrangian neighbourhood theorem \cite{Weinstein1971} to the setting of cosymplectic Lie algebroids.

\begin{theorem}[Cosymplectic Lagrangian neighbourhood theorem]\label{thm:lagrangiannbhd1}
	Let $\AA\Rightarrow M$ be a Lie algebroid equipped with two  $k$-cosymplectic structures $(\alpha_1, \dots, \alpha_k, \omega)$ and $(\tilde{\alpha}_1, \dots, \tilde{\alpha}_k, \tilde{\omega})$. Suppose that a transversal $i:L\hookrightarrow M$ for $\AA$ is a minimal Lagrangian transversal for both $k$-cosymplectic structures. 
	Then there exists a Lie algebroid isomorphism $(\tilde{\varphi}, \varphi): \AA\vert_{U}\to \AA\vert_{\tilde{U}}$ over open neighbourhoods $U$, $\tilde{U}$ of $L$ in $M$ such that the following hold:
	\begin{itemize}[noitemsep, topsep=0em]
		\item $\tilde{\varphi}^*\tilde{\alpha}_j=\alpha_j$ for $j=1, \dots, k$;  
		\item $\tilde{\varphi}^*\tilde{\omega}=\omega+\sum_{j=1}^k d(f^j\alpha_j)$ for some functions $f^j$. 
		\item $\tilde{\varphi}$ restricts to the identity on $i^!\AA\to L$.
	\end{itemize} 
	In particular, the Lie algebroid isomorphism $(\tilde{\varphi}, \varphi)$ is an $\AA$-Poisson isomorphism for the underlying Poisson structures.
\end{theorem}
We first look at  some consequences of this theorem. In the following, the map $p_*:(i^!\AA)^*\to L$ is the bundle projection, not to be confused with the projection map of the normal bundle $p_{NL}:NL\to L$ (cf. Remark \ref{rk:normalbundleimplicit}).

\begin{definition}
	Consider the Lie algebroid $p_*^!i^!\AA\times T\RR^k\Rightarrow(i^!\AA)^*\times \RR^k$, equipped with the standard $k$-cosymplectic structure $(dt_1, \dots, dt_k, \omega_{\can})$, where $\omega_{\can}$ is the canonical symplectic structure on $p_*^!i^!\AA$ pulled back along the map $p_*^!i^!\AA\times T\RR^k\to p_*^!i^!\AA$. This is the \textit{local model} of the $k$-cosymplectic structure around the minimal Lagrangian transversal $L$. 
\end{definition}
\begin{remark}
	The standard cosymplectic structure on the Lie algebroid $p_*^!i^!\AA\times T\RR^k\to (i^!\AA)^*\times \RR^k$ induces the linear Poisson structure on $(i^!\AA)^*\times \RR^k$ dual to the Lie algebroid $i^!\AA\oplus \RR^k$. 
\end{remark}

\begin{corollary}[Linearization]\label{cor:lagrangiannbhdlocalmodel}
	Let $(\AA, \alpha_1, \dots, \alpha_k, \omega) \Rightarrow M$ be a $k$-cosymplectic Lie algebroid and $i:L\hookrightarrow M$ a transversal for $\AA$ that is a minimal Lagrangian. Then there are neighbourhoods $U$ of $L$ in $M$ and $V$ of $L$ in $(i^!\AA)^*\times \RR^k$ and a Lie algebroid isomorphism $(\tilde{\varphi}, \varphi): \AA\vert_U \to p_*^!i^!\AA\times T\RR^k\vert_V$ such that 
	\begin{itemize}[noitemsep, topsep=0em]
		\item $\tilde{\varphi}^*dt_j=\alpha_j$ for $j=1, \dots, k$;  
		\item $\tilde{\varphi}^*\omega_{\mathrm{can}}=\omega+\sum_{j=1}^k d(f^j\alpha_j)$ for some functions $f^j$. 
		\item $\tilde{\varphi}$ restricts to the identity on $i^!\AA\to L$.
	\end{itemize} 
	In particular, $(\tilde{\varphi}, \varphi)$ is an isomorphism between the underlying Poisson algebroids, so that $\varphi$ is a Poisson linearization around $L$. 
\end{corollary}
\begin{remark}\label{rk:normalbundleimplicit}
	Implicitly, the cosymplectic structure is used to identify $NL\cong (i^!\AA)^*\times \RR^k$, giving rise to an isomorphism $p^!_*i^!\AA\times T\RR^k\cong p_{NL}^!i^!\AA$. 
\end{remark}
In the case that $k=0$, the theorem reduces to a Lagrangian neighbourhood theorem for symplectic Lie algebroids.

\begin{corollary}[Symplectic Lie algebroids]
	Let $(\AA, \omega)\Rightarrow M$ be a symplectic Lie algebroid and $i:L\hookrightarrow M$ a Lagrangian transversal. Then there is a local isomorphism of symplectic Lie algebroids
	\[
	(\AA, \omega)\xrightarrow[\mathrm{local}]{\cong} (p_*^!i^!\AA, \omega_{\mathrm{can}})
	\]
	that restricts to the identity on $i^!\AA\Rightarrow L$. 
\end{corollary}
Now follows a list of examples.
\begin{example}[Cosymplectic manifolds]
	When $\AA=TM$ is the usual tangent bundle, equipped with a $k$-cosymplectic structure, we find that the underlying Poisson structure around a Lagrangian submanifold $L\subset M$ contained in a single symplectic leaf is linearizable. The linearization is Poisson diffeomorphic to Poisson structure underlying the cosymplectic structure $(dt_1, \dots, dt_k, \omega_{\mathrm{can}})$ on $T^*L\times \RR^k$, where the flat map associated to the $k$-cosymplectic structure is used to identify $T^*L\times \RR^k$ with $NL$. 
\end{example}
\begin{example}[Log-symplectic manifolds]\label{ex:blagrangiannbhd}
	Let $(M, \pi)$ be a log-symplectic manifold \cite{guilleminmirandapires2014} with degeneracy locus $Z$. If $L$ is a Lagrangian submanifold that is transverse to $Z$, then $\tilde{Z}=Z\cap L$ is hypersurface in $L$, and there is a local log-symplectomorphism 
	\[
	(T_{\tilde{Z}}^*L, \omega_{\can})\xrightarrow[\mathrm{local}]{\cong} (M, \omega).
	\]
	This appeared in the PhD-thesis of Kirchoff-Lukat \cite{Kirchoff2018phd}. 
	\end{example}
\begin{example}[$b^k$-symplectic manifolds]	
Likewise, Theorem \ref{thm:lagrangiannbhd1} can be applied to the $b^k$-symplectic manifolds, introduced by Scott in \cite{scott2016}. A $b^k$-symplectic structure on a smooth log-manifold $(M, Z)$ is a Poisson structure $\pi$ for which $\wedge^{\mathrm{top}}\pi$ vanishes to order $k$ over $Z$. The $(k-1)$-jet of $\wedge^{\mathrm{top}}\pi$ can be used to define a Lie algebroid, called the \textit{$b^k$-tangent bundle} ${}^{b^k}T_ZM$ to which $\pi$ lifts non-degenerately. The anchor of this Lie algebroid surjects onto $TZ$, and thus $\pi$ is linearizable around a Lagrangian submanifold transverse to $Z$.
\end{example}	
\begin{example}[Elliptic Poisson manifolds]
A Poisson structure $\pi$ on $M$ is \textit{elliptic} when $\left(\wedge^{\mathrm{top}}\pi\right)^{-1}(0)=D$ is a codimension-2 submanifold and the Hessian of $\wedge^{\mathrm{top}}\pi$ along $D$ is non-degenerate \cite{Cavalcantigualtieri2017}. These type of Poisson structures appear naturally in the context of generalized complex structures that are almost everywhere symplectic. The elliptic Poisson structure lifts to a symplectic structure on the \textit{elliptic tangent bundle} $T_{|D|}M$. Since the anchor of this Lie algebroid surjects onto $TD$, the elliptic Poisson structure $\pi$ is linearizable around any Lagrangian submanifold transverse to $D$. This elliptic version of the Lagrangian neighbourhood theorem also appeared in \cite{Cavalcantigualtieri2017}.
\end{example}
\begin{remark}
	The examples above all fit into the general class of Poisson structure \textit{of divisor type}, studied by Klaasse in \cite{klaasse2017, klaasse2018}. The $\AA$-Lagrangian neighbourhood theorem applies to many of the examples in \cite{klaasse2018}.
\end{remark}
\begin{example}[Symplectic foliations]
	Let $(\FF, \omega)\Rightarrow M$ be a symplectic foliation. A Lagrangian submanifold $i:L\hookrightarrow M$ transverse to $\FF$ has a neighbourhood in $M$ that is isomorphic to a neighbourhood of the zero section in $(i^!\FF)^*=(\FF\cap TL)^*$ equipped with the symplectic foliation $(p_{(i^!\FF)^*}^!(i^!\FF), \omega_{\mathrm{can}})$. 
\end{example}

\begin{proof}[Proof of Theorem \ref{thm:lagrangiannbhd1}] The Moser Lemma \ref{lem:moserlemma} can not be applied directly, we first have to make the cosymplectic structures agree along $L$. To do so, we repeatedly apply the splitting theorem for Lie algebroids.
	
	First, we decompose the cosymplectic information along $L$ in terms of subbundles. Let $\FF=\ker(\alpha_1)\cap \dots \cap \ker(\alpha_k)$. The Reeb sections $R_i$ associated to cosymplectic structure span the subbundle $\ker\omega=\langle R_1, \dots, R_k\rangle $ complementary to $\FF$. Next, because $\omega$ is non-degenerate on $\FF$, we can choose a Lagrangian complement $\LL$ to $i^!\AA$ in $\FF\vert_L$. Put together, there is a decomposition
	\[
	\AA\vert_L= \FF\vert_L\oplus \langle R_1, \dots, R_k\rangle\vert_L=i^!\AA\oplus \LL\oplus \langle R_1, \dots, R_k\rangle\vert_L.
	\]
	The cosymplectic structure induces an isomorphism $\psi:(i^!\AA)^*\oplus \RR^k\to \LL\oplus \langle R_1, \dots, R_k\rangle$ uniquely determined by
	\[
	\omega(\psi(\eta, e_i), R_j)=\delta_{ij}, \quad \omega(\psi(\eta, e_i), v)=\eta(v) \mbox{ for $\eta\in (i^!\AA)^*$ and $v\in i^!\AA$,}
	\]
	where $e_1, \dots, e_k$ is the standard basis of $\RR^k$. 
	
	Similarly, from the cosymplectic structure $(\tilde{\alpha}_1, \dots, \tilde{\alpha}_k, \tilde{\omega})$ we obtain a decomposition $\AA\vert_L=i^!\AA\oplus \tilde{\LL}\oplus\langle \tilde{R}_1, \dots, \tilde{R}_k\rangle$ and an isomorphism $\tilde{\psi}:(i^!\AA)^*\oplus \RR^k\to \tilde{\LL}\oplus \langle\tilde{R}_1, \dots, \tilde{R}_k\rangle$. For notational purposes, we set $\mc{C}=\LL\oplus \langle R_1, \dots, R_k\rangle$ and $\tilde{\mc{C}}=\tilde{\LL}\oplus\langle \tilde{R}_1, \dots, \tilde{R}_k\rangle$.
	
	Let $\varphi=\tilde{\psi}\circ \psi^{-1}$, and define a bundle isomorphism
	\[
	\begin{gathered}
		\Psi:\AA\vert_L=i^!\AA\oplus \mc{C} \to i^!\AA\oplus \tilde{\mc{C}}=\AA\vert_L\\
		(v,c)\mapsto (v, \varphi(c)).
	\end{gathered}
	\]
	Clearly, $\Psi^*(\tilde{\omega}\vert_L)=\omega\vert_L$, $\Psi(\tilde{\alpha}_j\vert_L)=\alpha_j\vert_L$ and $\Psi\vert_{i^!\AA}=\id\vert_{i^!\AA}$. It remains to extend $\Psi$ to a Lie algebroid isomorphism in a neighbourhood of $L$. To do so, we descend on to the linearization of $\AA$.
	
	Let $p:NL\to L$, $q:\mc{C}\to L$ and $\tilde{q}:\tilde{\mc{C}}\to M$ be the bundle projections. Choose an isomorphism $(\tilde{\psi}, \psi):\AA\vert_U\to p^!i^!\AA$ from the Splitting Theorem \ref{thm:splitting}. Then the identifications $NL\cong N_\AA L\cong \mc{C}$ and $NL\cong N_\AA L\cong \tilde{\mc{C}}$ (as in Section \ref{subsubsec:liealgebroidtransversals}) together with the map $\varphi$ induce Lie algebroid isomorphisms as depicted by the following diagram.
	\[
	\begin{tikzcd}
		\AA\vert_U \arrow[d] & p^!i^!\AA \arrow[l, "\tilde{\psi}"'] \arrow[d] & q^!i^!\AA \arrow[d] \arrow[r, "T_{i^!\AA}\varphi"] \arrow[l, "\cong"'] & \tilde{q}^!i^!\AA \arrow[d] \arrow[r, "\cong"] & p^!i^!\AA \arrow[r, "\tilde{\psi}"] \arrow[d] & \AA\vert_U \arrow[d] \\
		U                    & NL \arrow[l, "\psi"']                      & \mc{C} \arrow[r, "\varphi"] \arrow[l, "\cong"']              & \tilde{\mc{C}} \arrow[r, "\cong"]                 & NL \arrow[r, "\psi"]                      & U                   
	\end{tikzcd}
	\]
	From the canonical identification
	\[
	q^!i^!\AA\vert_L\cong(i^!\AA)\oplus \ker Tq\vert_L\cong (i^!\AA)\oplus \mc{C},
	\] 
	and the fact that $T_{i^!\AA}\varphi\vert_L:(i^!\AA)\oplus \mc{C}\to (i^!\AA)\oplus \tilde{\mc{C}}$ equals $\Psi$, it follows that the composite, denoted $(\tilde{\Phi}, \tilde{\varphi}):\AA\vert_U\to \AA\vert_U$, extends $\Psi$.
	
	To complete the proof, we apply the cosymplectic Moser lemma to the cosymplectic structures $(\alpha_1, \dots, \alpha_k, \omega)$ and $(\tilde{\Phi}^*\tilde{\alpha}_1, \dots, \tilde{\Phi}^*\tilde{\alpha}_k, \tilde{\Phi}^*\tilde{\omega})$ on $\AA\vert_U$, that we have made to agree along $L$.
\end{proof}

\subsection{Dual integrations of triangular Lie bialgebroids}

For ordinary Poisson manifolds, the cotangent algebroid integrates to a (local) symplectic groupoid, which is always linearizable by Weinstein's Lagrangian neighbourhood theorem. On the other hand, if $(\mf{g}, \mf{g}^*)$ is a triangular Lie bialgebra, then a Poisson Lie group $(G^*, \Pi)$ integrating $(\mf{g}^*, \mf{g})$ is linearizable by a result of Alekseev and Meinrenken \cite{alekseevmeinrenken2013} (in fact, their result holds for coboundary Lie bialgebras). The following theorem is a simultaneous generalization of the two.

\begin{theorem}\label{thm:dualintegrationlinearizable}
Let $(\AA, \AA^*, \pi_\AA)$ be a triangular Lie bialgebroid, and $(\GG^*, \Pi)\toto (M, \pi)$ a Poisson groupoid integrating $(\AA^*, \AA)$. Then $(\GG^*, \Pi)$ is linearizable around $M$.  
\end{theorem}
	
Let us first consider the case that $\AA=TM$, in which case $\pi_{TM}=\pi$ is a Poisson structure on $M$. An integration $(\GG, \Pi)\toto (M, \pi)$ of the triangular Lie bialgebroid $(T^*M, TM)$ is actually a symplectic groupoid. The unit section of a symplectic groupoid is always Lagrangian, and thus Weinstein's Lagrangian neighbourhood theorem provides a linearization.

The proof of the general case is surprisingly similar. We only have to replace the tangent bundles in the proof above by appropriate Lie algebroids. More specifically, an $\AA$-version of the symplectic groupoid of a Poisson Lie algebroid and an $\AA$-version of Weinstein's Lagrangian neighbourhood theorem \ref{thm:lagrangiannbhd1} are needed for the proof.

The algebroid version of a symplectic groupoid is a symplectic LA-groupoid, that we will introduce now.

\subsubsection{Symplectic LA-groupoids}
Let $(\mc{V};\AA, \GG;M)$ be an LA-groupoid. A $\VV$-form $\Omega\in \Omega^k(\VV\Rightarrow \GG)$ is \textit{multiplicative} when $\tilde{m}^*\Omega=\tilde{\pr}_1^*\Omega+\tilde{\pr}_2^*\Omega$. A more elaborate discussion on multiplicative forms on LA-groupoids can be found in Appendix \ref{app:multiplicativeforms}. 
\begin{definition}
	A \textit{symplectic LA-groupoid} is an LA-groupoid $(\VV; \AA, \GG; M)$ with a multiplicative symplectic $\VV$-form $\Omega\in \Omega^2(\VV\Rightarrow \GG)$.  
	
	A \textit{local} symplectic LA-groupoid is a local LA-groupoid with a multiplicative symplectic germ. 
\end{definition} 
Because the symplectic form $\Omega$ is multiplicative, it induces an isomorphism of VB-groupoids $(\VV; \AA, \GG; M)\rightarrow (\VV^*; C^*, \GG; M)$. The Poisson structure $\Pi$ on $\GG$ induced by $\Omega$ is also multiplicative, by the diagram below. This makes $(\GG, \Pi)$ into a Poisson groupoid.
\begin{center}
	\begin{tabular}{ccc}
		\begin{tikzcd}
			\VV^* \arrow[d, shift left]\arrow[d, shift right] \arrow[r] & \GG \arrow[d, shift left]\arrow[d, shift right]\\
			C^*\arrow[r] & M.
			\end{tikzcd}
	
	&
$\xleftarrow{\Omega^\flat}$ & 
		\begin{tikzcd}
		\VV \arrow[d, shift left] \arrow[d, shift right] \arrow[r] & \GG \arrow[d, shift left]\arrow[d, shift right]\\
		\AA \arrow[r] & M
	\end{tikzcd}
\\
$\uparrow$ & & $\downarrow$
\\
\begin{tikzcd}
	T^*\GG \arrow[r] \arrow[d, shift left] \arrow[d, shift right] & \GG \arrow[d, shift left]\arrow[d, shift right]\\
	\AA^*\GG \arrow[r] & M
\end{tikzcd}
& $\xrightarrow{\Pi^\sharp}$ &
\begin{tikzcd}
	T\GG \arrow[r]\arrow[d, shift left]\arrow[d, shift right] & \GG \arrow[d, shift left]\arrow[d, shift right]\\
	TM \arrow[r] & M
\end{tikzcd}
	\end{tabular}
\end{center}
Here, $C$ is the core of $\VV$. The restriction of $\Omega^\flat$ to the unit space or the core gives rise to isomorphisms 
\[
\Omega^\flat\big\vert_\AA:\AA\to C^*, \quad \Omega^\flat\big\vert_C:C\to \AA^*,
\]
which are anti-dual by skew-symmetry of $\Omega$. By means of the core anchors $\rho_C^\AA:C\to \AA$ and $\rho^{C^*}_{\AA^*}=\left(\rho_C^\AA\right)^*$, we can define a bundle morphism
\[
\pi^\sharp_\AA:=\rho_C^\AA\circ \left( \Omega^\flat\big\vert_C\right)^{-1}= - \left(\Omega^\flat\big\vert_{\AA}\right)^{-1}\circ\rho^{C^*}_{\AA^*}:\AA^*\to \AA,
\]
which inherits skew-symmetry from $\Omega$. When we equip $\VV^*$ with the Lie algebroid structure inherited from $\Omega$ (Remark \ref{rk:exactbialgebroids}), it becomes an LA-groupoid whose core is the Lie algebroid structure $\AA^*\Rightarrow M$ with bracket induced by $\pi_\AA$. The map $\pi^\sharp_\AA$ is becomes a Lie algebroid morphism with respect to this structure, and therefore $\pi_\AA$ is an $\AA$-Poisson structure by Remark \ref{rk:exactbialgebroids}.
\begin{remark}
	The triangular bialgebroid $(\AA, \AA^*, \pi_\AA)$ is anchored to the bialgebroid $(\AA\GG, \AA^*\GG)$ of $(\GG, \Pi)$ as follows:
	\[
	\begin{tikzcd}
		\AA^*\arrow[r, "\pi_\AA^\sharp"]  \arrow[d] & \AA\\
		\AA\GG & \AA^*\GG\arrow[u] 
	\end{tikzcd}
	\]
	where $\AA^*\to \AA\GG$ arises from restricting $\rho_V$ to the core $\AA^*$. This means that the core anchor $\AA^*\to \AA\GG$ is a morphism of Lie bialgebroids $(\AA^*, \AA)\to (\AA\GG, \AA^*\GG)$. 
\end{remark}

\begin{example}\label{ex:symplecticLAgroupoids}
	Let $(\AA, \omega)\Rightarrow M$ be a symplectic Lie algebroid, and $\GG\toto  M$ an integration of $\AA$. Then the LA-groupoid 
	\[
	\begin{tikzcd}
		p^!\GG \arrow[r, Rightarrow] \arrow[d, shift left]\arrow[d, shift right] & \GG \arrow[d, shift left]\arrow[d, shift right]\\
		\AA\arrow[r, Rightarrow] & M
	\end{tikzcd}
	\]
	carries a multiplicative symplectic form given by $\tilde{s}^*\omega-\tilde{t}^*\omega$. The underlying Poisson groupoid integrates the triangular bialgebroid $(\AA, \AA^*, \omega^{-1})$. 
	
	More generally, when $(\BB, \omega)\Rightarrow M$ is a symplectic Lie algebroid that is anchored to $\AA$, and when $\GG\toto M$ is an integration of $\AA$, then $p^!_{\BB}\GG\Rightarrow \GG$ carries a symplectic form given by $\tilde{s}^*\omega-\tilde{t}^*\omega$. 
\end{example}

\subsubsection{The symplectic LA-groupoid integrating an $\AA$-Poisson structure}

Let $(\AA, \AA^*, \pi_\AA)$ be a triangular Lie bialgebroid. Since $\pi_\AA$ anchors $\AA^*$ to $\AA$, the pullback algebroid $p_*^!\AA$ becomes a double Lie algebroid. Moreover, it comes with a canonical symplectic form $\omega_\can$ on the vertical side structure $p^!_*\AA\Rightarrow \AA^*$. Let $\VV\to \GG^*$ be a (possibly local) source 1-connected VB-groupoid integrating the horizontal structure. We wish to integrate $\omega_\can$ to $\VV$.
\begin{proposition}\label{prop:omegacanim}
	Let $(\AA, \AA^*, \pi_\AA)$ be a triangular Lie bialgebroid. Then $\omega_\can$ is infinitesimally multiplicative with respect to the horizontal Lie algebroid structure $p^!_* \AA \Rightarrow \AA$. 
\end{proposition}

\begin{proof}
	Linearity of $\omega_{\can}$ was already part of Lemma \ref{lem:canonicalformprolongation}. To prove that $\omega_{\can}$ is an IM-form, we must show that the associated map $c_{\omega_{\can}}:p^!_*\AA\oplus_{\AA^*} p^!_*\AA\to \RR$ satisfies the cocycle condition
	\[
	c_{\omega_{\can}}([X, Y])=\LL_X(c_{\omega_{\can}}(Y))-\LL_Y(c_{\omega_{\can}}(X))
	\]
	for sections $X, Y\in \Gamma(p^!_*\AA\oplus_{\AA^*} p^!_*\AA\to \AA\oplus \AA)$. It is enough to verify this on the generators, which are of the form
	\[
	T^2_\AA\alpha=(T_\AA\alpha, T_\AA\alpha), \quad \widehat{\alpha}^{(1)}=(\widehat{\alpha}, T_\AA 0), \quad  \widehat{\alpha}^{(2)}=(T_\AA0, \widehat{\alpha}), 
	\]
	for $\alpha\in \Gamma(\AA^*)$, on which the Lie bracket is given by
	\[
	[T^2_\AA\alpha, T^2_\AA\beta]=T^2_\AA[\alpha, \beta], \quad [T^2_\AA\alpha, \widehat{\beta}^{(i)}]=\widehat{[\alpha, \beta]}^{(i)}, \left[ \widehat{\alpha}^{(i)}, \widehat{\beta}^{(j)}\right]=0.
	\]
	
	\textit{Linear-linear.} Let $\alpha, \beta\in \Gamma(\AA^*)$ and $v, w\in \Gamma(\AA)$. By Lemma \ref{lem:canonicalformprolongation}, we have
	\[
	c_{\omega_{\can}}([T^2_\AA\alpha, T^2_\AA\beta])(v, w)=d[\alpha, \beta](v, w).
	\]
	Using the definition of the bracket, we find that
	\[
	d[\alpha, \beta](v, w)=d\LL_{\pi_\AA^\sharp(\alpha)}(\beta)(v, w)-d\LL_{\pi_\AA^\sharp(\beta)}(\alpha)(v, w).
	\]
	On the other hand, note that the anchor $\oplus_{\AA^*} \varrho_{\AA^*}:p^!_*\AA\oplus_{\AA^*}p^!_*\AA\to T(\AA\oplus\AA)$, defined in Section \ref{subsec:prolongationsoverLiealgebroids}, sends $T^2_\AA\alpha$ to the complete lifts $\left(\widetilde{\pi^\sharp_\AA(\alpha)}, \widetilde{\pi^\sharp_\AA(\alpha)}\right)$, whose flow is given by $T(\varphi_t\oplus\varphi_t)$, with $(\varphi_t, \tilde{\varphi}_t):\AA\to \AA$ the flow of $\pi^\sharp(\alpha)$. Therefore, we find by Lemma \ref{lem:canonicalformprolongation}
	\begin{align*}
		\LL_{T^2_\AA\alpha}\left(c_{\omega_\can}(T^2_\AA\beta)\right)&=\ddt\big\vert_{t=0} \left(c_{\omega_{\can}}(T^2_\AA\beta)\right)(\varphi_t\circ v, \varphi_t\circ w)\\
		&=\ddt\big\vert_{t=0}\varphi^*_t(d\beta)(v, w)\\
		&=\LL_{\pi^\sharp_\AA(\alpha)}(d\beta)(v, w)=d(\LL_{\pi^\sharp_\AA(\alpha)}(\beta))(v, w).
	\end{align*}
	The cocycle condition now easily follows.
	
	\textit{Linear-core.} From Lemma \ref{lem:canonicalformprolongation}, it directly follows that
	\[
	c_{\omega_{\can}}([T^2_\AA\alpha, \widehat{\beta}^{(1)}])(v, w)=-\left\langle [\alpha, \beta], w\right\rangle.
	\]
	On the other hand, Let $\varphi_t$ be the flow of $\pi^\sharp(\alpha)$ as before. Then
	\[
	\LL_{T^2_\AA\alpha}\left(c_{\omega_{\can}}(\widehat{\beta}^{(1)})\right)(v, w)=\ddt\big\vert_{t=0}c_{\omega_{\can}}(\widehat{\beta}^{(1)})(\varphi_t\circ v, \varphi_t\circ w)=-\left\langle\LL_{\pi^\sharp_\AA(\alpha)}(\beta),w\right\rangle,
	\]
	and, 
	\[
	\LL_{\widehat{\beta}^{(1)}}(c_{\omega_{\can}}(T^2_\AA\alpha))(v, w)=\ddt\big\vert_{t=0}(d\alpha)(v+t\pi^\sharp_\AA(\beta), w)=(d\alpha) (\pi^\sharp_\AA(\beta), w).
	\]
	because $\varrho_{\AA^*}(\widehat{\beta}^{(1)})=((\pi^\sharp_\AA(\beta))^\uparrow, 0)\in \mf{X}(\AA\oplus\AA)$ (so its flow is just a translation by $\pi^\sharp_\AA(\beta)$). It follows that
	\begin{align*}
		\left(\LL_{T^2_\AA\alpha}(c_{\omega_{\can}}(\widehat{\beta}^{(1)}))-\LL_{\widehat{\beta}^{(1)}}(c_{\omega_{\can}}(T^2_\AA\alpha))\right)(v, w)&=-\left\langle \LL_{\pi^\sharp_\AA(\alpha)}(\beta), w\right\rangle +\left\langle\iota_{\pi^\sharp_\AA(\beta)}d\alpha, w\right\rangle\\
		&=-\left\langle[\alpha, \beta], w\right\rangle.
	\end{align*}
	For $\widehat{\beta}^{(2)}$, the argument is similar. 
	
	\textit{Core-core.} It is clear that 
	\[
	c_{\omega_{\can}}\left(\widehat{\alpha}^{(i)}(v), \widehat{\beta}^{(j)}(w)\right)=0.
	\]
	Also, using a similar calculation from above
	\[
	\LL_{\beta^{(1)}}\left(c_{\omega_{\can}}(\alpha^{(2)})\right)(v, w)=\ddt\big\vert_{t=0}\langle \alpha, v+t\pi^\sharp_\AA(\beta)\rangle=\langle \alpha, \pi^\sharp_\AA(\beta)\rangle,
	\] 
	and similar equations for the other combinations. The cocycle condition for the core-core pairing is now equivalent to $\pi_\AA$ being skew-symmetric. 
\end{proof}

\begin{corollary}\label{cor:dualintegrationsymplecticLAgroupoid}
Let $(\AA, \AA^*, \pi_\AA)$ be a triangular Lie bialgebroid, and $\GG^*\toto M$ a source 1-connected (local) groupoid integrating $\AA^*$. Suppose that 
\[
\begin{tikzcd}
	\VV \arrow[r, shift left]\arrow[r, shift right] \arrow[d, Rightarrow] & \AA \arrow[d, Rightarrow]\\
	\GG^* \arrow[r, shift left] \arrow[r, shift right] & M
\end{tikzcd}
\]
is a (local) LA-groupoid integrating $p_*^!\AA$. Then $\VV$ carries a unique $\VV$-symplectic structure making it into a (local) symplectic LA-groupoid, and such that the underlying Poisson structure on $\GG^*$ integrates the bialgebroid $(\AA^*, \AA)$. 	
\end{corollary}
\begin{proof}
Since $\GG^*$ is source 1-connected, so is $\VV^*$. By Proposition \ref{prop:omegacanim}, the canonical form $\omega_{\can}$ on $p^!_*\AA$ is IM and symplectic, and thus it integrates by Corollary \ref{cor:multiplicativecochainiso} to a unique $\VV$-symplectic form $\Omega$. Since $\omega_{\can}$ induces the Poisson structure on $\AA^*$ dual to $\AA$, $\Omega$ induces the Poisson structure $\Pi$ on $\GG^*$ integrating the Lie bialgebroid $(\AA^*, \AA)$. 
\end{proof}

\begin{remark}
	Integrability of the algebroid $p^!_*\AA\Rightarrow \AA$ always implies integrability of $\AA^*\Rightarrow M$, but the converse is not immediate \cite{brahiccabreraortiz2018}. We haven't investigated integrability of the prolongation Lie algebroid, which is the reason why we work with local LA-groupoids. 
\end{remark}

\begin{proof}[Proof of Theorem \ref{thm:dualintegrationlinearizable}]
		We appeal to Corollary \ref{cor:dualintegrationsymplecticLAgroupoid}, and assume that (at least locally) the Poisson structure $\Pi$ is induced by the symplectic LA-groupoid $(\VV, \Omega)$ integrating the symplectic double Lie algebroid $(p^!_*\AA, \omega_{\can})$.
	We show that $M\subset \GG^*$ is a Lagrangian transversal for $(\VV, \Omega)$. 
	
	\textit{Tranversality.} The anchor $p^!_*\AA\to T\AA^*$ is certainly transverse to the zero section. In fact, it is the identity on the core. Therefore, the anchor of the integration $\rho_\VV:\VV\to T\GG^*$ is also the identity on the core. The core of $T\GG^*$ is transverse to $TM$, hence $M$ is a transversal for $\VV$. 
	
	\textit{Lagrangian. } let $u:M\hookrightarrow \GG^*$ be the inclusion. Since $M$ is always coisotropic, we deduce from Lemma \ref{lem:coisotropictransversal} that it is enough to show that $\dim u^!\VV=\frac{1}{2} \dim \VV$. But this is clear because $u^!\VV=\rho_\VV^{-1}(TM)=\AA$, as $\rho_\VV$ is the identity on the cores. 
	
	Linearizability follows from Theorem \ref{thm:lagrangiannbhd1}. 
\end{proof}
\subsection{Integrations of cosymplectic Lie algebroids}

Let $(\AA, \AA^*, \pi_\AA)$ be a triangular Lie bialgebroid, and $\GG\toto M$ an integration of $M$. Recall from \cite{liuxu1996} that $\lar{\pi_\AA}-\rar{\pi_\AA}$ defines a Poisson structure on $\GG$ integrating the Lie bialgebroid $(\AA, \AA^*)$. 

Suppose there exists an $\AA$-cosymplectic structure $(\alpha_1, \dots, \alpha_k, \omega)$ inducing $\pi_\AA$. Then, on the Lie algebroid $p^!\GG\Rightarrow \GG$, the cosymplectic structure
\[
\left(t^*\alpha_1, \dots, t^*\alpha_k, s^*\alpha_1, \dots, s^*\alpha_k, s^*\omega-t^*\omega\right)
\]
induces the Poisson structure $\lar{\pi_\AA}-\rar{\pi_\AA}$ on $\GG$. Note that $u:M\hookrightarrow \GG$ is \emph{not} a minimal Lagrangian for this cosymplectic structure, because $u^*t^*\alpha_j=\alpha_j$. To overcome this, define $\mathscr{F}=\cap_{j=1}^k\ker  t^*\alpha_j$; this is a wide subalgebroid of $p^!\GG$. The cosymplectic structure on $p^!\GG$ descends to a cosymplectic structure
\[
\left(s^*\alpha_1\big\vert_{\mathscr{F}}, \dots, s^*\alpha_k \big\vert_{\mathscr{F}}, (s^*\omega-t^*\omega)\big\vert_{\mathscr{F}}\right)
\] 
on $\mathscr{F}$. The unit space $u:M\hookrightarrow \GG$ is still transverse to $\mathscr{F}$, pulling it back to 
\[
u^!\mathscr{F}=\FF=\cap_{j=1}^k \ker \alpha_j.
\]
In this situation, $M$ becomes a minimal Lagrangian transversal for the cosymplectic Lie algebroid $\mathscr{F}$! After an application of Theorem \ref{thm:lagrangiannbhd1}, we have proven the following result.
\begin{theorem}\label{thm:cosymplecticpairgroupoidlinearizable}
	Let $(\AA, \AA^*, \pi_\AA)$ be a triangular Lie bialgebroid of cosymplectic type. Let $(\GG, \Pi)\toto M$ be a Poisson groupoid integrating $(\AA, \AA^*)$. Then $(\GG, \Pi)$ is linearizable around $M$. 
\end{theorem}
\begin{remark}
	The fundamental difference between Theorem \ref{thm:dualintegrationlinearizable} and Theorem \ref{thm:cosymplecticpairgroupoidlinearizable} is that the former considers integrations of $\AA^*$, while the latter concerns integrations of $\AA$. They coincide when $\pi_\AA$ is non-degenerate. 
\end{remark}
\begin{remark}[On cosymplectic groupoids]
There is no direct connection between the groupoids integrating Poisson algebroids of cosymplectic type and the cosymplectic groupoids as considered by Djiba and Wade in \cite{djibawade2015}. Both the one-form and two-form of a cosymplectic groupoid are multiplicative, which is not the case above, even after adjustments. It is interesting to note that cosymplectic groupoids are always linearizable as Poisson groupoids by the cosymplectic Lagrangian neighbourhood Theorem \ref{thm:lagrangiannbhd1}, because the unit section is automatically a minimal Lagrangian.
\end{remark}

\subsection{(Non)-linearizability of the pair Poisson groupoid}

For a Poisson manifold $(M, \pi)$, the product $(M\times M, \pi\times (-\pi))$ is rarely linearizable around the diagonal. 

\begin{theorem}\label{thm:nonlinearizable}
Let $(M, \pi)$ be a Poisson manifold. If $(M\times M, \pi\times (-\pi))$ is linearizable around the diagonal, then $\pi$ must be regular. 
\end{theorem}
At first sight the theorem can be quite unexpected, since it implies the following. 
\begin{corollary}
	Let $\mf{g}$ be a Lie algebra. Then $(\mf{g}^*\times \mf{g}^*, \pi_\mf{g}\times (-\pi_\mf{g}))$ is linearizable around the diagonal if and only if $\mf{g}$ is abelian. 
\end{corollary}

\begin{remark}
	When $\pi$ is regular, and of $k$-cosymplectic type, then $(M\times M, \pi\times (-\pi))$ is linearizable around the diagonal according to Theorem \ref{thm:cosymplecticpairgroupoidlinearizable}.
\end{remark}
Yet, our proof of Theorem \ref{thm:nonlinearizable} is very geometric and intuitive. 

\begin{proof}[Proof of theorem \ref{thm:nonlinearizable}] 
	Suppose $\pi$ is not regular. We have to show that any neighbourhood of the diagonal of in $(M\times M, \pi\times(-\pi))$ can not be Poisson diffeomorphic to a neighbourhood of the zero section in $(TM, \pi_{\mathrm{lin}})$. To this end, we define, for a Poisson manifold $(M, \pi)$, the sets of regular and singular points of $M$:
	\begin{align*}
	\Reg_\pi(M)&=\{ x\in M: \mbox{$\pi$ is regular in a neighbourhood of $x$}\},\\ \Sing_\pi(M)&=M\setminus \Reg_\pi(M).
	\end{align*}
	\begin{claim}{1}
	The set $\Reg_\pi(M)$ is an open and dense subset of $M$, while $\Sing_\pi(M)$ is nowhere dense and closed. 
	\end{claim}

\begin{claim}{2}
	Denote the Poisson structure $\pi\times (-\pi)$ on $M\times M$ by $\Pi$. We have 
	\begin{align*}
	\Sing_\Pi(M\times M)&=\Sing_\pi(M)\times M\cup M\times \Sing_\pi(M).\\
	\Sing_{\pi_{\mathrm{lin}}}(TM)&=\left( TM\right)\vert_{\Sing_\pi(M)}.
	\end{align*}
\end{claim}
\begin{proof}[Proof of Claim 2]
	It is clear that $\Reg_\Pi(M\times M)= \Reg_\pi(M)\times \Reg_\pi(M)$, so the first equality follows. The second follows from Weinstein's splitting theorem for Poisson structures \cite{weinstein1983}. There are coordinates $(q^i, p_i, y^i)$ on $M$ for which $\pi$ is given by (using Einstein's summation convention)
	\[
	\pi=\deldel{q^i}\wedge\deldel{p_i} + \theta^{ij}(y)\deldel{y^i}\wedge\deldel{y^j},
	\]
	where $\theta^{ij}(0)=0$. 
	Denoting the linear functions $u^i, v_i, w^i$ on $TM$ dual to $dq^i, dp_i, dy^i$, respectively, the Poisson bracket on $TM$ is determined by
	\begin{align*}
	\{ u^i, p_j\}&=\delta^i_j, &  \{v_i, q^j\}&=-\delta^j_i, &
	\{ w^i, y^j\}&=\theta^{ij}, & \{ w^i, w^j\}&=\frac{\del \theta^{ij}}{\del y^k} w^k,
	\end{align*}
and all the other coordinate combinations vanish. 
It follows that the rank of $\pi_{\mathrm{lin}}$ is given by
\[
\operatorname{rk} \pi_{\mathrm{lin}}(y)=2\operatorname{rk}\pi\vert_{y=0}+\operatorname{rk} \begin{pmatrix}
	0 & \theta^{ij}(y)\\
	\theta^{ji}(y) & \frac{\del\theta^{ij}}{\del y^k}(y) w^k
\end{pmatrix},
\]
which jumps whenever $\operatorname{rk}\left(\theta^{ij}\right)_{ij}$ does: around a singular point in $M$, and fixing coordinates $w^k$, the matrix $\left(\frac{\del \theta^{ij}}{\del y^k}(y) w^k\right)_{ij}$ does not decrease rank in a neighbourhood of $y=0$, while $(\theta^{ij})_{ij}$ jumps rank in every neighbourhood. The second equality follows.   
\end{proof}

Let $S\subset M$ be any subset. The \textit{formal tangent space} of $S$ at a point $s\in S$ is defined to be 
\[
\mc{T}_s S=\operatorname{span}\left\{ \dot{\gamma}(0)\big\vert \mbox{$\gamma\in C^\infty(\RR, M)$ s.t. $\gamma(0)=x$ and $\gamma(\RR)\subset S$}\right\}. 
\]
Note that if $\varphi:M\to M$ is a diffeomorphism with $\varphi(S)=S$, then $T_s\varphi$ sends $\mc{T}_s S$ isomorphically to $\mc{T}_{\varphi(s)}S$.
\begin{claim}{3}
Let $U\subset \RR^n$ be an open subset with non-empty boundary. Then there exists an $x\in \del U$ such that $\mc{T}_x(\del U)\neq T_x \RR^n$. In fact, if $V\subset \RR^n$ is any open such that $V\cap \del U\neq \emptyset$, then $x$ can be chosen to be in $V$.
\end{claim}
It can often happen that $\mc{T}_x (\del U)=T_x \RR^n$ for some $x\in \del U$. For instance, take $U$ to be the complement of the union of all coordinate axes in $\RR^n$. Then $\mc{T}_0 (\del U)=\RR^n$. Claim 3 ensures that this is not the case for a dense subset of the boundary.
\begin{proof}[Proof of Claim 3]
	Choose any $y\in U$, and let $r=d(y, \del U)>0$. There exists $x\in \del U$ such that $r=d(y, x)$. Let $\gamma:\RR\to \del U$ be a curve with $\gamma(0)=x$. Then $d(y, \gamma(t))$ attains its minimum at $t=0$, from which follows that $\dot{\gamma}(0)$ is tangent to the sphere centred around $y$ with radius $r$. For the last part, we pick $\tilde{x}\in V\cap \del U$, pick $\epsilon>0$ such that $B(\tilde{x}, 2\epsilon)\subset V$, and then pick $y\in B(\tilde{x}, \epsilon)\cap U$. This ensures that $x$ as above is in $V$.
\end{proof}

\begin{claim}{4}
	Let $x\in \Sing_\pi(M)$. Then
	\begin{align*}
	\mc{T}_{(x, x)} \Sing_\Pi (M\times M)&=T_{(x, x)}(M\times M)\\ 
	\mc{T}_x \Sing_{\pi_{\mathrm{lin}}}(TM)&=\mc{T}_x \Sing_{\pi}(M)\times T_xM,
	\end{align*}
where we used the canonical identification $T_x (TM)=T_x M\times \ker Tp_{TM}\vert_M=T_xM\times T_x M$ for $x\in M$.
\end{claim}
\begin{proof}[Proof of Claim 4]
	This easily follows from the description of the sets in question in Claim 2.
\end{proof}
Now we are in a position to complete the argument. Suppose $(M, \pi)$ is not regular. Then, according to Claim 3, there exist an element $x\in \Sing_\pi(M)=\del \Reg_\pi(M)$ for which $\mc{T}_x \Sing_\pi(M)\neq T_x M$. But then the sets $\mc{T}_x \Sing_{\pi_{\mathrm{lin}}}(TM)$ and $\mc{T}_{(x,x)}\Sing_\Pi(M\times M)$ can not be isomorphic by Claim 4. Hence, there exists no (local) Poisson diffeomorphism $(TM, \pi_{\mathrm{lin}})\dashrightarrow (M\times M, \pi\times (-\pi))$ that extends the diagonal embedding. 
\end{proof}

\appendix

\section{Multiplicative forms on LA-groupoids}\label{app:multiplicativeforms}
We recall the definitions of several combined structures encountered in this paper. For more details, we refer to \cite{bursztyncabrerahoyo2016, mackenzie1992, mackenzie1997, mackenzie2005}. 
\begin{itemize}[noitemsep, topsep=0em]
	\item  A Lie bialgebroid $(\AA, \AA^*)$ is a Lie algebroid $\AA\Rightarrow M$ with a Lie algebroid structure on $\AA^*\Rightarrow M$ such that
	\[
	d_{\AA^*}[v, w]_\AA=[d_{\AA^*}v, w]_\AA +[v, d_{\AA^*}w]_\AA, \quad \mbox{ for all $v, w\in \Gamma(\AA)$.}
	\] 
	\item A \textit{double vector bundle} $(D; A, B; M)$ is a square
	\[
	\begin{tikzcd}
	D \arrow[r] \arrow[d]  & B \arrow[d] \\
	A \arrow[r] & M
	\end{tikzcd}
\]
with a vector bundle structure associated to each edge of the diagram for which the scalar multiplication of the vector bundle structure on $D\to B$ and $D\to A$ commute. The \textit{core} $C$ of $D$ is the intersection of the kernels of the projections $D\to A$ and $D\to M$. It becomes a vector bundle $C\to M$ by using the scalar multiplication induced by either $D\to A$ or $D\to B$.
\item A \textit{VB-algebroid} is a double vector bundle $(D; A, B;M)$ with Lie algebroid structures $D\Rightarrow B$ and $A\Rightarrow M$ for which the scalar multiplication on $D\to A$ is a Lie algebroid morphism. In this case, we say that the VB-algebroid $(D;A, B;M)$ has a \textit{horizontal} Lie algebroid structure. One can also consider a double vector bundle $(D;A, B;M)$ with Lie algebroid structures $D\Rightarrow A$ and $B\Rightarrow M$ for which the scalar multiplication on $D\to B$ is a Lie algebroid morphism. This is also called a VB-algebroid, but now it has a \textit{vertical} Lie algebroid structure.
\item A \textit{VB-groupoid} $(V; A, \GG; M)$ 
(with a vertical groupoid structure) is a square
\[
\begin{tikzcd}
	V \arrow[d, shift left]\arrow[r]\arrow[d, shift right]& \GG \arrow[d, shift left] \arrow[d, shift right] \\
	A \arrow[r] & M
\end{tikzcd}
\]
for which the horizontal arrows are vector bundle structures, and the vertical arrows are groupoid structures, that are compatible in the sense that the scalar multiplication of the vector bundle $V\to \GG$ is a groupoid morphism from $\VV\toto A$ to $\VV\toto A$. One can define a VB-groupoid with a horizontal groupoid structure in a similar way.
\item An \textit{LA-groupoid} is a $VB$-groupoid $(\VV; \AA, \GG; M)$ with Lie algebroid structures $\VV\Rightarrow \GG$ and $\AA\Rightarrow M$ for which the groupoid structure maps of $\VV\toto \AA$ are Lie algebroid morphisms. 
\item A \textit{double Lie algebroid} is a double vector bundle $(\mc{D}; \AA, \BB; M)$ with a horizontal and vertical VB-algebroid structures
\[
\begin{tikzcd}
	\mc{D} \arrow[r, Rightarrow] \arrow[d]  & \BB \arrow[d] \\
	\AA \arrow[r, Rightarrow] & M,
\end{tikzcd}
\begin{tikzcd}
	\mc{D} \arrow[r] \arrow[d, Rightarrow]  & \BB \arrow[d, Rightarrow] \\
	\AA \arrow[r] & M,
\end{tikzcd}
\]
that are compatible in the following way. The anchors of the VB-algebroid structures must be Lie algebroid morphisms for the other, and that under the duality of double vector bundles, the pair $(D_A^*, D_B^*)$ is a Lie bialgebroid over the core dual $C^*$. Because this definition is quite lengthy, we think of double Lie algebroids in this paper as the infinitesimal counterpart of a (local) LA-groupoid, as explained in \cite{mackenzie1997} (or \cite{bursztyncabrerahoyo2016} via the duality between Lie algebroids and linear Poisson structures). This makes some proofs shorter and more conceptual. The differentiation procedure is recalled in section \ref{sec:multiplicativeformsLAgroupoids}.
\end{itemize}
\begin{remark}
	The notation $(D;A, B;M)$, whether a double vector bundle or an LA-groupoid, always refers to the following \textit{arrangement} of the diagram
	\[
	\begin{tikzcd}
		D \arrow[r, no head, dashed] \arrow[d, no head, dashed] & B \arrow[d, no head, dashed]\\
		A \arrow[r, no head, dashed] & M.
	\end{tikzcd}
	\]
	Whether the groupoid/algebroid structure is vertical or horizontal will be clear from the context or otherwise explicitly mentioned.
\end{remark}
\subsection{Multiplicative tensors on VB-groupoids}
We give a brief account on the differentiation and integration procedure of multiplicative tensors on VB-groupoids, as discussed in \cite{egea2016}. 
\subsubsection{Multiplicative functions}
Let $\GG$ be a Lie groupoid, with Lie algebroid $\AA\GG=\ker Ts\vert_M$. Given $f\in C^\infty(\GG)$, one can differentiate $f$ to obtain
\[
\AA f: \AA\GG\to\RR, \quad v\mapsto \left\langle df\vert_M, v\right \rangle.
\]
Clearly, $\AA f$ is linear and thus corresponds to a section of $\AA^*\GG$. 
\begin{definition}
	A function $f:\GG\to \RR$ on a (local) Lie groupoid $\GG\toto M$ is multiplicative when $f(gh)=f(g)+f(h)$ for all $g, h\in \GG$ composable.
\end{definition}
Given a multiplicative function $f$ on $\GG$, it differentiates to a Lie algebroid morphism $\AA f: \AA\GG\to \RR$, where $\RR\to \{*\}$ is regarded as a Lie algebroid with trivial bracket.
\begin{definition}
	Let $\AA\Rightarrow M$ be a Lie algebroid. A function $f:\AA\to \RR$ is called \textit{IM (infinitesimally multiplicative)} when it is linear and a Lie algebroid morphism. Concretely, this means that $f$ is linear and satisfies the cocycle condition:
	\[
	f([v, w])=\LL_{\rho(v)}(f(w))-\LL_{\rho(w)}(f(v))
	\]
	for sections $v, w\in \AA$. Equivalently, $d_{\AA\GG}  f=0$, when the linear function $f$ is regarded as a section of $\AA^*$. 
\end{definition}

\subsubsection{Multiplicative forms} By a $k$-form on a vector bundle $E\to M$ we mean an element of $\Gamma(\wedge^kE^*)$. Let $(\mc{V}; \AA, \GG; M)$ be a VB-groupoid (with vertical groupoid structure). A $k$-form $\tau\in \Gamma(\wedge^k\mc{V}^*)$ induces a function
\[
c_{\tau}:\oplus^k_\GG \mc{V}\to \RR
\]
that is $k$-linear and skew-symmetric. The space $\oplus^k_\GG \VV$ is naturally VB-groupoid over $\oplus^k\AA$. In the following, the groupoid structure maps of a VB-groupoid are decorated with a tilde.
\begin{proposition}
	Let $(\VV; \AA, \GG; M)$ be an VB-groupoid (with a vertical groupoid structure) and $\tau\in \Gamma(\wedge^k\VV^*)$. Then the following are equivalent:
	\begin{itemize}[noitemsep, topsep=0em]
		\item $\tau$ is multiplicative, i.e. $\tilde{m}^*\tau=\tilde{\pr}_1^*\tau+\tilde{\pr}_2^*\tau$;
		\item $c_\tau:\oplus^k\VV\to \RR$ is multiplicative;
		\item $\tau^\flat: \oplus^{k-1}\VV\to \VV^*$ is a groupoid morphism. 
	\end{itemize}
\end{proposition}

One can differentiate a multiplicative form $\tau\in \Gamma(\wedge^k \VV^*)$ by differentiating $c_\tau$ to a Lie algebroid morphism
\[
\AA c_\tau:\AA\left(\oplus^k_\GG \VV\right)\cong \oplus^k_{\AA\GG}\AA\VV\to \RR,
\]
which is $k$-linear, skew-symmetric. Therefore, it must be induced by some linear $k$-form $\AA\tau\in \Gamma(\wedge^k(\AA\VV)^*_{\AA\GG})$.

\begin{definition}
	Let $(\mc{D}; \AA, \BB; M)$ be a VB-algebroid (with a vertical Lie algebroid structure). A form $\tau\in \Gamma(\wedge^k(\mc{D})^*_\BB)$ on the horizontal vector bundle $\mc{D}\to \mc{B}$ is \textit{IM (infinitesimally multiplicative)} with respect to the \textit{vertical} Lie algebroid structure when 
	\[
	c_\tau: \oplus_\BB^k \mc{D} \to \RR
	\]
	is an IM function on the Lie algebroid $\oplus^k_\BB \mc{D}\Rightarrow \oplus^k \AA$. 
\end{definition}
\begin{theorem}[\cite{egea2016}, Theorem 2.35]\label{thm:integrationIMforms}
	If $(\VV; \AA, \GG; M)$ is a source 1-connected (local) VB-groupoid integrating the VB-algebroid $(\AA\VV; \AA, \AA\GG; M)$. The assignment $\tau\mapsto \AA\tau$ gives rise to a one-to-one correspondence between (germs of) multiplicative forms on $\VV \rightarrow \GG$ and IM forms on $\AA\VV\rightarrow\AA\GG$. 
\end{theorem}

\begin{remark}
	The definitions and results in this section are formulated for vertical VB-groupoids and VB-algebroids. By flipping all the diagrams, one do the exact same things with horizontal VB-groupoids and VB-algebroids. For example, an IM-form for a horizontal VB-algebroid $(\mc{D}; \AA, \BB;M)$ is a section $\tau\in \Gamma(\wedge^k (\mc{D})^*_\AA)$ (so a form on the vector bundle $\mc{D}\to \AA$) for which 
	\[
	c_\tau: \oplus^k_\BB \mc{D} \to \RR
	\]
	is an IM function on the Lie algebroid $\oplus^k_\AA \mc{D}\Rightarrow \oplus^k \BB$.
\end{remark}

\subsection{Multiplicative forms on LA-groupoids}\label{sec:multiplicativeformsLAgroupoids}

Let $(\VV;\AA,\GG;M)$ be an LA-groupoid. What we are after is the interaction between the Lie algebroid structure on $\VV\Rightarrow \GG$, the Lie algebroid structure $\AA\VV\Rightarrow \AA\GG$ and the differentiation process of the multiplicative forms. 

The main result of this section is the following theorem. 
\begin{theorem}\label{thm:differentiationcochainmap}
	Let $(\VV; \AA, \GG; M)$ be a (local) LA-groupoid differentiating to the double Lie algebroid $(\AA\VV; \AA, \AA\GG; M)$. The map $\tau \mapsto \AA\tau$ intertwines the Lie algebroid differentials:
	\[
	d_{\AA\VV}\circ \AA=\AA\circ d_\VV.
	\]
	Here, $d_\VV$ is the Lie algebroid differential on $\Omega^\bullet (\VV\Rightarrow \GG)$ and $d_{\AA\VV}$ is the Lie algebroid differential on $\Omega^\bullet (\AA\VV\Rightarrow \AA \GG)$. 
\end{theorem}
For an LA-groupoid $(\VV;\AA, \GG; M)$, we denote by $\Omega^\bullet_M(\VV\Rightarrow \GG)$ the space of multiplicative forms. Likewise, the space of IM forms on $(\mc{D}; \AA, \BB; M)$ is denoted by $\Omega^\bullet_{IM}(\mc{D}\Rightarrow \AA)$. 
\begin{corollary}\label{cor:multiplicativecochainiso}
	If $(\VV; \AA, \GG; M)$ is a source 1-connected LA-groupoid, then the differentiation map gives an isomorphism of cochain complexes from $(\Omega_M^\bullet(\VV\Rightarrow \GG), d_\VV)$ to $(\Omega_{IM}^\bullet(\AA\VV\Rightarrow \AA\GG), d_{\AA\VV})$. 
\end{corollary}

Before we proceed to the proof, we must first describe the structure of the Lie algebroid $\AA\VV\Rightarrow \AA\GG$, as was done in \cite{mackenzie1997}.
\begin{definition}
	Let $(\VV;\AA, \GG;M)$ be an LA-groupoid, that is possibly local.  A \textit{star section} is an element $\xi\in \Gamma(\VV\to \GG)$ for which there exists a section $X\in \Gamma(\AA)$ satisfying $\tilde{s}\circ \xi=X\circ s$ and $\xi\circ u=\tilde{u}\circ X$. 
\end{definition}
The Lie bracket of two star sections is again a star section, and any star section $(\xi, X)$ of $\VV$ differentiates to a linear section $(\AA\xi, X)$ of $\AA\VV\to \AA\GG$. Additionally, every core section $c\in \Gamma(C)$ lifts to a section $\widehat{c}\in \Gamma(\AA\VV\to \AA\GG)$ via (cf. equation (\ref{eq:corelift}))
\[
\widehat{c}(v)=\AA(0)(v)+\overline{c(p_{\AA\GG}(v))}. 
\] 
\begin{proposition}[\cite{mackenzie1997}, Proposition 1.11]\label{prop:generatorsprolongation}
	The space of sections of $\AA\VV\to \AA\GG$ is generated (as a $C^\infty(\AA\GG)$-module) by sections of the form $(\AA\xi, X)$ for $(\xi, X)$ a star section and lifts of core sections $\widehat{c}$ for $c\in \Gamma(C)$. 
\end{proposition}
We now describe the Lie algebroid structure. First, recall that the canonical involution $J_\GG:T(T\GG)\to T(T\GG)$ restricts to an isomorphism of double vector bundles $j_\GG: \AA(T\GG)\to T(\AA\GG)$. 

\textbf{The anchor map.} The anchor of the prolonged Lie algebroid structure on $\AA\VV$ is given by $\varrho=j_\GG^{-1}\circ \AA(\rho_\VV):\AA \VV\to T(\AA\GG)$. It is double linear, and covers the Lie anchor map $\rho_\AA:\AA\to TM$ on the side bundle, the identity on $\AA\GG$ on the other side bundle and the core morphism is the core anchor $\rho_C^{\AA\GG}:C\to \AA\GG$. Also, notice that $\varrho(\widehat{c})=\rho_C^{\AA\GG}(c)^\uparrow$ for sections $c\in \Gamma(C)$. 

\textbf{The Lie bracket.} The Lie bracket is specified on the derivatives of star sections $\AA\xi$ and lifts $\widehat{c}$ of core sections $c\in \Gamma(C)$, and extended via the Leibniz rule. Let $\xi_1, \xi_2$ be star sections and $c_1, c_2\in \Gamma(C)$. 
\begin{itemize}[noitemsep, topsep=0em]
	\item (Star-star.) The Lie bracket of $\AA\xi_1$ and $\AA\xi_2$ is set to
	\[
	[\AA\xi_1, \AA\xi_2]=\AA([\xi_1, \xi_2]).
	\]
	\item (Star-core.) Since $\xi_1$ is $\tilde{s}$-related to $X_1$, and $\rar{c_1}$ is $\tilde{s}$ related to $0$, it follows that $\tilde{s}([\xi_1, \rar{c_1}])=0$. Therefore, we can set
	\[
	[\AA\xi_1, \widehat{c}_1]=\widehat{D_{\xi_1}(c_1)},
	\]
	with $D_{\xi_1}(c)=[\xi_1, \rar{c_1} ]\circ u\in \Gamma(C)$. 
	\item (Core-core.) The bracket of $\widehat{c}_1$ and $\widehat{c}_2$ is simply set to zero.
\end{itemize} 
One of the main results in \cite{mackenzie1997} is that this Lie algebroid structure makes $(\AA\VV; \AA\GG, \AA;M)$ into a double Lie algebroid. 
\begin{proof}[Proof of theorem \ref{thm:differentiationcochainmap}] First we consider functions.
	\begin{claim}{1}
		Let $f:\GG\to \RR$, be a function, $(\xi, X)$ a star section and $c\in \Gamma(C)$. Then 
		\begin{align*}
			\LL_{\AA\xi}(\AA f)&=\AA(\LL_\xi(f)),& \quad \LL_{\widehat{c}}(\AA f)&=p_{\AA\GG}^*u^*(\LL_{\rar{c}}(f)),\\
			\LL_{\AA\xi}(p_{\AA\GG}^* u^*f)&=p_{\AA\GG}^*u^*(\LL_\xi(f)), & \LL_{\widehat{c}}(p_{\AA\GG}^*u^*f)&=0.
		\end{align*}
	\end{claim}
	\begin{proof}[Proof of claim]
		For the first part, let $\varphi_\epsilon$ be the flow of $\rho_V(\xi)$ covering the flow of $\rho_\AA(X)$. Then $\AA(\varphi_\epsilon)$ is the flow of $\varrho(\AA(\xi))$ \cite{mackenziexu1998}. Therefore, given $v\in \AA\GG$ with $v=\dot{\gamma}_0$, we have
		\begin{align*}
			\AA(\LL_\xi(f))(v)&=\frac{d}{d\lambda}\big\vert_{\lambda=0} \LL_\xi(f)(\gamma_\lambda)=\frac{d}{d\lambda}\big\vert_{\lambda=0} \frac{d}{d\epsilon}\big\vert_{\epsilon=0}f\circ \varphi_\epsilon\circ\gamma_\lambda\\
			&= \frac{d}{d\epsilon}\big\vert_{\epsilon=0} \left\langle \AA(f), \AA\varphi_\epsilon(v)\right\rangle =\LL_{\AA\xi}(\AA f)(v).
		\end{align*}
		Next, the flow $\tilde{\varphi}_\epsilon$ of $\rho_\AA(X)$ is covered both by $\AA\varphi_\epsilon$ and $\varphi_\epsilon$. Therefore
		\[
		\LL_{\AA\xi}(p^*_{\AA\GG}u^*f)=p^*_{\AA\GG}\left(\LL_{\rho(X)}u^* f\right)=p^*_{\AA\GG}u^*\left(\LL_\xi(f)\right).
		\] 
		Recalling that $\varrho(\widehat{c})=(\rho_C^{\AA\GG}(c))^\uparrow$, we have
		\[
		\LL_{\hat{c}}(\AA f)=p_{\AA\GG}^*\left\langle \AA f, \rho^{\AA\GG}_C(c)\right\rangle=p^*_{\AA\GG}u^*(\LL_{\rar{c}}(f)).
		\]
		The last identity is obvious because $\varrho(\widehat{c})$ is vertical.
	\end{proof}
	\begin{claim}{2}
		Let $\tau\in \Omega^k(\VV\Rightarrow \GG)$ be multiplicative, $(\xi_1,X_1), \dots, (\xi_k, X_k)$ star sections, and $c_1, c_2\in \Gamma(C)$. Then
		\begin{align*}
			(\AA\tau)(\AA\xi_1, \dots, \AA\xi_k)&=\AA(\tau(\xi_1, \dots, \xi_k)), \\ (\AA\tau)(\AA\xi_1, \dots \AA\xi_{k-1}, \widehat{c}_1)&=p_{\AA\GG}^* u^*\left(\tau(\xi_1, \dots, \xi_{k-1}, \rar{c_1})\right), \\ \iota_{\widehat{c}_1}\iota_{\widehat{c}_2}(\AA\tau)&=0. 
		\end{align*}
	where $u:M\hookrightarrow \GG$ is the unit map.
	\end{claim}
	\begin{proof}[Proof of claim]
		The first holds via the identification $\AA(\oplus^k_\GG\VV)=\oplus^k_{\AA\GG}\AA\VV$: if $v=\dot{\gamma}_0\in \AA\GG$, then 
		\[
		(\AA\xi_1, \dots, \AA\xi_k)(v)=\frac{d}{d\epsilon}\vert_{\epsilon=0}(\xi_1, \dots, \xi_k)(\gamma_\epsilon),
		\]
		so it follows that
		\[
		\AA c_\tau(\AA\xi_1, \dots, \AA\xi_k)(v)=\frac{d}{d\epsilon}\big\vert_{\epsilon=0} c_\tau(\xi_1(\gamma_\epsilon)), \dots,\xi_k(\gamma_\epsilon))=\AA(\tau(\xi_1, \dots, \xi_k))(v).
		\]
		For the second and third identities, let $c\in \Gamma(c)$. Note that, given $v=\dot{\gamma}_0$ with $s(\gamma_\epsilon)=\gamma_0$, the map $(\lambda, \epsilon)\mapsto 0_{\gamma_\epsilon} \cdot \lambda c(\gamma_0)$ is well-defined for all $\lambda, \epsilon$. Therefore,
		\[
		\frac{d}{d\epsilon}\big\vert_{\epsilon=0}0_{\gamma_\epsilon}\cdot \epsilon c(\gamma_0)=\frac{d}{d\epsilon}\big\vert_{\epsilon=0} 0_{\gamma_\epsilon}\cdot 0_{\gamma_0} + \frac{d}{d\epsilon}\big\vert_{\epsilon=0} 0_{\gamma_0}\cdot \epsilon c(\gamma_0)=\AA(0)(v)+\overline{c(q_{\AA\GG}(v))}.
		\]
		where we used that $0_{u(m)}=\tilde{u}(0_m)$. It follows that
		\[
		\hat{c}(v)=\frac{d}{d\epsilon}\big\vert_{\epsilon=0} 0_{\gamma_\epsilon}\cdot \epsilon c(\gamma_0)=\frac{d}{d\epsilon}\big\vert_{\epsilon=0} \epsilon \left(0_{\gamma_\epsilon}\cdot c(\gamma_0)\right). 
		\]
		Now the second and third identity follow from $k$-linearity of $\tau$. 
	\end{proof}
	
	Finally, we check that $d(\AA\tau)=\AA(d_\VV\tau)$ on the generators, making use of the previous claims. 
	
	\textit{(Star-star.)} Let $\xi_0, \dots, \xi_k$ be star sections. Then
	\begin{align*}
		d_{\AA\VV}(\AA\tau)(\AA\xi_0, \dots, \AA\xi_k)&=\sum_{i=0}^k(-1)^i \LL_{\AA\xi_i}(\AA\tau(\AA\xi_0, \overset{\widehat{\AA\xi}_i}{\dots}, \AA\xi_k))\\
		&\phantom{=}+\sum_{i<j}(-1)^{i+j}\AA\tau([\AA\xi_i, \AA\xi_j], \AA\xi_0, \overset{\widehat{\AA\xi}_i}{\dots}, \overset{\widehat{\AA\xi}_j}{\dots}, \AA\xi_k)\\
		&=\AA\left( d_\VV\tau(\xi_0, \dots, \xi_k)\right)=\AA\left(d_\VV\tau\right)(\AA\xi_0, \dots, \AA\xi_k).
	\end{align*} 

	\textit{(Star-core.)} Let $\xi_0, \dots, \xi_{k-1}$ be star sections and $c\in \Gamma(C)$. Then
	\begin{align*}
		d_{\AA\VV}(\AA\tau)(\AA\xi_0, \dots, \AA\xi_{k-1}, \hat{c})&=\sum_{i=0}^{k-1}(-1)^i\LL_{\AA\xi}\left( \AA\tau(\AA\xi_0, \overset{\widehat{\AA\xi}_i}{\dots}, \AA\xi_{k-1}, \hat{c})\right)\\
		&\phantom{=}+(-1)^k \LL_{\hat{c}}\left(\AA\tau(\AA\xi_0, \dots, \AA\xi_{k-1})\right)\\
		&+ \sum_{i<j}(-1)^{i+j}\AA\tau(\AA[\xi_i, \xi_j], \AA\xi_0, \overset{\widehat{\AA\xi}_i}{\dots},\overset{\widehat{\AA\xi}_j}{\dots}, \AA\xi_{k-1}, \hat{c})\\
		&\phantom{=}+\sum_i(-1)^{i+k}\AA\tau(\widehat{D_{\xi_i}(c)}, \AA\xi_0, \overset{\widehat{\AA\xi}_i}{\dots}, \AA\xi_{k-1}).
	\end{align*}
	We inspect this term by term.
	\begin{itemize}[noitemsep, topsep=0em]
		\item[1.] Using claim 2 and then claim 1, the first term evaluates to 
		\begin{align*}
			\sum_{i=0}^{k-1}(-1)^i p^*_{\AA\GG}u^*\left(\LL_{\widehat{\xi}_i}\left(\tau(\xi_0, \overset{\xi_i}{\dots}, \xi_{k-1}, \rar{c})\right)\right).
		\end{align*}
		\item[2.] For the second one, we get
		\[
		(-1)^kp^*_{\AA\GG}u^*\left(\LL_{\rar{c}}\left(\tau(\xi_0, \dots, \xi_{k-1})\right)\right).
		\]
		\item[3.] The third term equals
		\[
		\sum_{i<j}(-1)^{i+j}p^*_{\AA\GG}u^*\left( \tau([\xi_i, \xi_i],\xi_0, \overset{\widehat{\xi}_i}{\dots}, \overset{\widehat{\xi}_j}{\dots}, \xi_{k-1}, \rar{c})\right)
		\]
		\item[4.] Similarly, we have for the fourth term
		\[
		\sum_{i}(-1)^{i+k}p^*_{\AA\GG}u^* \left(\tau([\xi_i, \rar{c}], \xi_0, \overset{\widehat{\xi}_i}{\dots}, \xi_{k-1})\right).
		\]
	\end{itemize}
	It follows that
	\begin{align*}
		d_{\AA\VV}(\AA\tau)(\AA\xi_0, \dots, \AA\xi_{k-1}, \hat{c})&=p^*_{\AA\GG}u^* \left( (d_\VV \tau)(\xi_0, \dots, \xi_{k-1}, \rar{c})\right)\\
		&=\AA(d_\VV\tau)(\AA\xi_0, \dots, \AA\xi_{k-1}, \hat{c}).
	\end{align*}
	
	\textit{(Core-core).} If $c_1, c_2\in \Gamma(C)$, then it can be checked by the Koszul formula and the claims that
	\[
	\iota_{\widehat{c_1}}\iota_{\widehat{c_2}}d_{\AA\VV}(\AA\tau)=0.
	\]
	Since $\iota_{\widehat{c_1}}\iota_{\widehat{c_2}}\AA(d_\VV\tau)$ also vanishes, the result follows.
\end{proof}

 \small
\bibliographystyle{abbrv}

\begin{thebibliography}{10}
	
	\bibitem{alekseevmeinrenken2013}
	A.~Alekseev and E.~Meinrenken.
	\newblock Linearization of {P}oisson {L}ie group structures.
	\newblock {\em Journal of Symplectic Geometry}, 14(1):227 -- 267, 2016.
	
	\bibitem{bursztyncabrerahoyo2016}
	H.~Bursztyn, A.~Cabrera, and M.~{del Hoyo}.
	\newblock Vector bundles over {L}ie groupoids and algebroids.
	\newblock {\em Advances in Mathematics}, 290:163 -- 207, 2016.
	
	\bibitem{Bursztyn2016Splitting}
	H.~Bursztyn, H.~Lima, and E.~Meinrenken.
	\newblock Splitting theorems for {P}oisson and related structures.
	\newblock {\em Journal f{\"u}r die reine und angewandte Mathematik (Crelles
		Journal)}, 2019:281 -- 312, 2016.
	
	\bibitem{brahiccabreraortiz2018}
	A.~Cabrera, O.~Brahic, and C.~Ortiz.
	\newblock Obstructions to the integrability of $\mathcal{VB}$-algebroids.
	\newblock {\em Journal of Symplectic Geometry}, 16(2):439–483, 2018.
	
	\bibitem{Cavalcantigualtieri2017}
	G.~R. Cavalcanti and M.~Gualtieri.
	\newblock Stable generalized complex structures.
	\newblock {\em Proceedings of the London Mathematical Society},
	116(5):1075–1111, 2017.
	
	\bibitem{conn1985}
	J.~F. Conn.
	\newblock Normal forms for smooth {P}oisson structures.
	\newblock {\em Annals of Mathematics}, 121(3):565--593, 1985.
	
	\bibitem{crainicfernandes2011}
	M.~Crainic and R.~L. Fernandes.
	\newblock A geometric approach to {C}onn's linearization theorem.
	\newblock {\em Annals of Mathematics}, 173(2):1121--1139, 2011.
	
	\bibitem{crainicmarcut2012}
	M.~Crainic and I.~Marcut.
	\newblock A normal form theorem around symplectic leaves.
	\newblock {\em Journal of Differential Geometry}, 92:417--461, 11 2012.
	
	\bibitem{djibawade2015}
	S.~A. Djiba and A.~Wade.
	\newblock On cosymplectic groupoids.
	\newblock {\em Comptes Rendus Mathematique}, 353(9):859 -- 863, 2015.
	
	\bibitem{egea2016}
	L.~G. Egea.
	\newblock {VB}-groupoids cocycles and their applications to multiplicative
	structures.
	\newblock {\em Ph.D. thesis, Instituto de Matem\'atica Pura e Aplicada}, 2016.
	
	\bibitem{ginzburgweinstein1992}
	V.~Ginzburg and A.~Weinstein.
	\newblock Lie-{P}oisson structure on some {P}oisson {L}ie groups.
	\newblock {\em Journal of the American Mathematical Society}, 5:445--453, 1992.
	
	\bibitem{guilleminmirandapires2014}
	V.~Guillemin, E.~Miranda, and A.~R. Pires.
	\newblock Symplectic and {P}oisson geometry on b-manifolds.
	\newblock {\em Advances in Mathematics}, 264:864 -- 896, 2014.
	
	\bibitem{Kirchoff2018phd}
	C.~Kirchoff-Lukat.
	\newblock Aspects of generalized geometry: Branes with boundary, blow-ups,
	brackets and bundles.
	\newblock {\em Ph.D. thesis, University of Cambridge}, 2018.
	
	\bibitem{klaasse2017}
	R.~Klaasse.
	\newblock Geometric structures and {L}ie algebroids.
	\newblock {\em Ph.D. thesis, Utrecht University}, 2017.
	
	\bibitem{klaasse2018}
	R.~Klaasse.
	\newblock Poisson structures of divisor type,
	\href{https://arxiv.org/abs/1811.04226}{arXiv:1811.04226}, 2018.
	
	\bibitem{leonmarreromartinez2005}
	M.~d. León, J.~C. Marrero, and E.~Martínez.
	\newblock Lagrangian submanifolds and dynamics on {L}ie algebroids.
	\newblock {\em Journal of Physics A: Mathematical and General},
	38(24):R241–R308, 2005.
	
	\bibitem{sjamaarpreparation}
	Y.~Lin, Y.~Loizides, R.~Sjamaar, and Y.~Song.
	\newblock Symplectic reduction and {D}arboux-{M}oser-{W}einstein theorems for
	symplectic {L}ie algebroids.
	\newblock in preparation.
	
	\bibitem{liuxu1996}
	Z.-J. Liu and P.~Xu.
	\newblock Exact {L}ie bialgebroids and {P}oisson groupoids.
	\newblock {\em Geometric {\&} Functional Analysis}, 6(1):138--145, 1996.
	
	\bibitem{mackenzie1992}
	K.~Mackenzie.
	\newblock Double {L}ie algebroids and second-order geometry, {I}.
	\newblock {\em Advances in Mathematics}, 94(2):180 -- 239, 1992.
	
	\bibitem{mackenzie1997}
	K.~Mackenzie.
	\newblock Double {L}ie algebroids and second-order geometry, {II}.
	\newblock {\em Advances in Mathematics}, 154(1):46--75, 2000.
	
	\bibitem{mackenzie2005}
	K.~Mackenzie.
	\newblock {\em General Theory of Lie Groupoids and Lie Algebroids}.
	\newblock London Mathematical Society Lecture Note Series. Cambridge University
	Press, 2005.
	
	\bibitem{mackenziexu1994}
	K.~Mackenzie and P.~Xu.
	\newblock Lie bialgebroids and Poisson groupoids.
	\newblock {\em Duke Mathematical Journal}, 73(2): 415-452, 1994.
	
	\bibitem{mackenziexu1998}
	K.~Mackenzie and P.~Xu.
	\newblock Classical lifting processes and multiplicative vector fields.
	\newblock {\em Quarterly Journal of Mathematics}, 49:59--85, 1998.
	
	\bibitem{scott2016}
	G.~Scott.
	\newblock {The geometry of $b^k$-manifolds}.
	\newblock {\em Journal of Symplectic Geometry}, 14:71--95, 2016.
	
	\bibitem{smilde2021liegroups}
	W.~Smilde.
	\newblock Lie groups of {P}oisson diffeomorphisms, \href{https://arxiv.org/abs/2108.11490}{arXiv:2108.11490}, 2021.
	
	\bibitem{vaisman1994}
	I.~Vaisman.
	\newblock {\em Lectures on the {G}eometry of {P}oisson {M}anifolds}.
	\newblock Birkh\"auser Basel, 1994.
	
	\bibitem{vorobiev2000}
	Y.~Vorobiev.
	\newblock Coupling tensors and {P}oisson geometry near a single symplectic
	leaf.
	\newblock {\em Banach Center Publications}, 54, 2000.
	
	\bibitem{Weinstein1971}
	A.~Weinstein.
	\newblock Symplectic manifolds and their {L}agrangian submanifolds.
	\newblock {\em Advances in Mathematics}, 6(3):329--346, 1971.
	
	\bibitem{weinstein1981}
	A.~Weinstein.
	\newblock {Symplectic geometry}.
	\newblock {\em Bulletin (New Series) of the American Mathematical Society},
	5(1):1 -- 13, 1981.
	
	\bibitem{weinstein1983}
	A.~Weinstein.
	\newblock {The local structure of Poisson manifolds}.
	\newblock {\em Journal of Differential Geometry}, 18(3):523 -- 557, 1983.
	
	\bibitem{weinstein1988}
	A.~Weinstein.
	\newblock Coisotropic calculus and {P}oisson groupoids.
	\newblock {\em Journal of the Mathematical Society of Japan}, 40(4):705--727,
	1988.
	
	\bibitem{wittephd}
	A.~Witte.
	\newblock Between generalized complex and {P}oisson geometry.
	\newblock {\em Ph.D. thesis, Utrecht University}, 2021.
	
\end{thebibliography}

\end{document}